\documentclass[3p]{scrartcl}

\usepackage{titling}
\usepackage[hmarginratio=1:1,top=32mm,left=15mm,columnsep=1cm]{geometry} %
\usepackage{authblk}    
\usepackage[utf8]{inputenc} 
\usepackage[T1]{fontenc}    
\usepackage{hyperref}       
\usepackage{url}            
\usepackage{microtype}
\usepackage{graphicx}
\usepackage{array}
\usepackage[svgnames,table,xcdraw]{xcolor} 
\usepackage{tabularx}
\usepackage{amsmath,mathdots,amsthm}
\usepackage{amssymb}
\usepackage{amsfonts}
\usepackage{xcolor}
\usepackage{bm}
\usepackage{soul}
\usepackage{float}
\usepackage{booktabs, multirow}
\usepackage{dutchcal} 
\usepackage{fourier}
\usepackage[scr=boondox,  
            cal=esstix]   
           {mathalpha}

\hypersetup{
	colorlinks=true,                          
	linkcolor=blue, 
	citecolor=DarkRed, 
	urlcolor=black  
} 
\graphicspath{{Figures/}}

\usepackage[backend=bibtex,giveninits=true,sorting=none,url=false,doi=true,eprint=false,isbn=false,
backref,backrefstyle=none,maxbibnames=99]{biblatex}
\DefineBibliographyStrings{english}{%
	backrefpage = {Cited on p\adddot},%
	backrefpages = {Cited on pp\adddot}%
}
\bibliography{biblio}

\renewbibmacro{in:}{%
	\ifboolexpr{%
		test {\ifentrytype{article}}%
	}{}{\printtext{\bibstring{in}\intitlepunct}}%
}

\newtheorem{theorem}{\textbf{Theorem}}[section]

\newtheorem*{rmrk*}{Remark}

\newtheorem*{definition*}{Definition}

\newtheorem*{thm*}{Theorem}

\newtheorem*{prp*}{Proposition}

\newcommand{\halb}{\frac{1}{2}}
\newcommand{\quart}{\frac{1}{4}}
\newcommand{\A}{\mathbf{A}} 
\newcommand{\B}{\mathbf{B}} 
\newcommand{\E}{\mathbf{E}} 
\newcommand{\D}{\mathbf{D}} 

\renewcommand{\H}{\mathbf{H}} 
\newcommand{\I}{\mathbf{I}} 
\newcommand{\M}{\mathbf{M}} 
\newcommand{\Y}{\mathbf{Y}} 
\newcommand{\p}{\mathbf{p}} 
\newcommand{\q}{\mathbf{q}} 
\newcommand{\f}{\mathbf{f}} 
\newcommand{\x}{\mathbf{x}} 
\newcommand{\pd}{\partial} 
 
\newcommand{\err}{\mathcal{r}}
\newcommand{\normal}{\mathbf{n}} 
\newcommand{\FD}{{^*\hspace{-0.75mm}F}}

\newcommand{\en}{\mathcal{E}}

\begin{document}

\title{\textbf{Variational derivation and compatible discretizations of the Maxwell-GLM system}}


\author[1]{Michael Dumbser\footnote{Corresponding author: \href{mailto:michael.dumbser@unitn.it}{michael.dumbser@unitn.it}}$^{,}$}
\author[1]{Alessia Lucca}
\author[1]{Ilya Peshkov}
\author[1]{Olindo Zanotti}
\affil[1]{\small Department of Civil, Environmental and Mechanical Engineering, University of Trento, Via Mesiano 77, 38123 Trento, Italy
}
\date{}



\maketitle

\begin{abstract}
\noindent\small
We present a novel variational derivation of the Maxwell-GLM system, which
augments the original vacuum Maxwell equations via a generalized Lagrangian
multiplier approach (GLM) by adding two supplementary acoustic subsystems
and which was originally introduced by Munz \textit{et al.} for purely
numerical purposes in order to treat the divergence constraints of the
magnetic and the electric field in the vacuum Maxwell equations within
general-purpose and non-structure-preserving numerical schemes for
hyperbolic PDE. Among the many mathematically interesting features of the
model are: i) its symmetric hyperbolicity, ii) the extra conservation law
for the total energy density and, most importantly, iii) the very peculiar
combination of the basic differential operators, since both, curl-curl and
div-grad combinations are mixed within this kind of system. A similar
mixture of Maxwell-type and acoustic-type subsystems has recently been also
forwarded by Buchman \textit{et al.} in the context of a reformulation of
the Einstein field equations of general relativity in terms of tetrads. This
motivates our interest in this class of PDE, since the system is by itself
very interesting from a mathematical point of view and can therefore serve
as useful prototype system for the development of new structure-preserving
numerical methods. Up to now, to the best of our knowledge, there exists
neither a rigorous variational derivation of this class of hyperbolic PDE
systems, nor do exactly energy-conserving and  asymptotic-preserving schemes
exist for them.  
The objectives of this paper are to derive the Maxwell-GLM system from an
underlying variational principle, show its consistency with Hamiltonian
mechanics and special relativity, extend it to the general nonlinear case
and to develop new exactly energy-conserving and asymptotic-preserving
finite volume schemes for its discretization. 
\end{abstract}

\section{Introduction}

Due to their mathematical elegance, since their discovery the Maxwell equations of electromagnetism \cite{MaxwellOrg} have inspired a vast amount of research in theoretical physics as well as in pure and applied mathematics. The first compatible numerical scheme for the discretization of the Maxwell equations in the time domain is due to Yee \cite{Yee66}, who introduced a suitable staggered mesh for the definition of the discrete electric and magnetic field in combination with compatible differential operators that guarantee that the discrete divergence of the discrete magnetic field is always zero and that the discrete divergence of the electric field vanishes in the absence of charges. Starting from the Yee scheme several different and more general numerical frameworks have been developed to guarantee the divergence-free property exactly at the discrete level, namely via the use of so-called mimetic finite differences \cite{HymanShashkov1997,Margolin2000,Lipnikov2014}, compatible finite volume schemes \cite{BalsaraCED,BalsaraKaeppeli,HazraBalsara} and compatible finite elements \cite{Nedelec1,Nedelec2,Hiptmair,Arnold2006,Monk,Alonso2015}. Much less schemes have been developed that preserve the curl-free property of a vector field, see e.g. \cite{JeltschTorrilhon2006,BalsaraCurlFree,SIGPR,SIST,Dhaouadi2023NSK}. Instead of imposing the divergence constraints on the electric and the magnetic field exactly, Munz \textit{et al.} in \cite{MunzCleaning} have introduced a very interesting approach called generalized Lagrangian multiplier (GLM) divergence cleaning, that resembles conceptually the concept of artificial compressibility applied in the numerical solution of the incompressible Navier-Stokes equations \cite{ChorinArtComp}. The idea of the GLM divergence cleaning is to admit divergence errors and to augment the original Maxwell equations by two additional acoustic subsystems so that the numerical divergence errors in the magnetic and electric field, which are generated by not exactly compatible general purpose schemes, do not accumulate locally but are rather transported away with an artificial cleaning speed $c_h$ via the acoustic subsystems that are coupled to the original Maxwell equations. A similar approach has recently also been forwarded for the numerical treatment of curl errors in curl-free vector fields, see \cite{dumbser2020glm,SHTCSurfaceTension,HyperbolicDispersion,Dhaouadi2022,TwoFluidDG}, leading to a similar combination of acoustic and Maxwell-type systems. The mathematical structure of the augmented Maxwell-GLM system is very intriguing and contains an interesting combination of curl-curl and div-grad operators, which is usually not present in classical PDE systems of continuum physics. 

Based on the groundbreaking work of Godunov \cite{God1961,God1961_2024eng}, which established for the first time a rigorous connection between underlying variational principle, thermodynamic compatibility and symmetric hyperbolicity in the sense of Friedrichs \cite{FriedrichsSymm}, a vast class of PDE systems of continuum physics has been subsequently studied and developed by Godunov and Romenski and collaborators in the last decades, the so-called class of symmetric hyperbolic and thermodynamically compatible (SHTC) systems, see \cite{God1972MHD,GodMHD_2024eng,GodRom1972,GodRom2003,Rom1998,RomenskiTwoPhase2007,RomenskiTwoPhase2010,PeshRom2014,PRD-Torsion2018,SHTC-GENERIC-CMAT} and references therein. The systems which fall into this class range from the compressible Euler and MHD equations over the equations of nonlinear hyperelasticity to the equations describing compressible multi-phase flows, superfluid helium and continuum mechanics with torsion. However, so far, in this very large class of systems no mixed coupling of curl-curl and div-grad operators was present, which is typical for the Maxwell-GLM system studied in this paper.  Yet, as will be shown later, the Maxwell-GLM system actually falls into the class of SHTC systems and can be derived from an underlying variational principle. This paper therefore extends the so far known class of SHTC systems.  

A very interesting system that has indeed a similar structure to the Maxwell-GLM system studied in this paper is the tetrad formulation of Buchman \textit{et al.} of the Einstein field equations of general relativity, see \cite{Buchman1,Buchman2}. For certain gauge choices the principal part of their PDE system reduces to a large set of Maxwell-type equations that are coupled with acoustic subsystems. Therefore, we believe that the Maxwell-GLM system studied in this paper is an interesting mathematical prototype system for other, more complex, field theories and is also suitable for the development of new structure-preserving numerical methods. 
The main objectives of this paper are therefore the following: 
\begin{enumerate}
	\item provide a new and rigorous derivation of the Maxwell-GLM system from an underlying variational principle;
	\item show the symmetric hyperbolicity of the system and the related additional extra conservation law for the total energy density; 
	\item study the asymptotic behaviour of the Maxwell-GLM system in the case when the cleaning speed $c_h$ tends to infinity;  
	\item extend the Maxwell-GLM system to more general nonlinear Lagrangian densities; this together with the variational principle and the extra conservation law allows to conclude that the Maxwell-GLM system falls indeed into the wider class of SHTC systems established by Godunov and Romenski and thus enlarges the known class of SHTC systems by a new one with different structure;  
	\item show the connection to Hamiltonian mechanics by providing a Poisson bracket from which the system can be also derived; 
	\item show the compatibility of this new variational principle with special relativity;
	\item develop new structure-preserving schemes for the Maxwell-GLM system and study their behaviour.
\end{enumerate} 
The rest of this paper is therefore organized as follows. In Section \ref{sec.model} we present the augmented Maxwell-GLM system together with its extra conservation law for the total energy density, we show its symmetric hyperbolicity, study the behaviour of the model in the limit when the cleaning speed tends to infinity and present a novel and rigorous variational derivation of this kind of PDE system. The only similar attempt we are aware of was documented in \cite{GLMLagrangian} but it is rather different from the variational approach presented in this paper. In Section \ref{sec.sr} we show the compatibility of the presented variational principle with special relativity, showing the connection with the Maxwell tensor and the Faraday tensor and that the Lagrangian density is a relativistic scalar invariant. In Section \ref{sec.numerics} we present two compatible discretizations. The first scheme is a semi-discrete finite volume method on collocated grids that preserves the total energy exactly and which also trivially extends to the nonlinear case. The second scheme is a vertex-based staggered semi-implicit scheme that preserves the basic vector calculus identities $\nabla \cdot \nabla \times \A = 0$ and $\nabla \times \nabla \phi = 0$ exactly on the discrete level and which is also exactly totally energy conservative. Some numerical results are presented in Section \ref{sec.results}. The paper closes with some concluding remarks and an outlook to future research in Section \ref{sec.conclusion}.      

\section{The augmented Maxwell-GLM system}
\label{sec.model} 

\subsection{Governing equations and extra conservation law}

The Maxwell-GLM system of Munz \textit{et al.} \cite{MunzCleaning},
which consists of the vacuum Maxwell equations augmented by two acoustic subsystems reads 
\begin{eqnarray}
	\partial_t \B + c_0 \nabla \times \E + c_h \nabla \phi &=& 0, \label{eqn.B} \\ 
	\partial_t \phi  + c_h \nabla \cdot \B  &=& 0. \label{eqn.phi}   \\	
	\partial_t \E - c_0 \nabla \times \B + c_h \nabla \psi &=& 0, \label{eqn.E} \\ 
\partial_t \psi  + c_h \nabla \cdot \E  &=& 0, \label{eqn.psi} 
\end{eqnarray}
In the following, we will refer to \eqref{eqn.B}-\eqref{eqn.psi} also as the \textit{Maxwell-Munz} system. 
Here, $c_0 \in \mathbb{R}^+$ is the vacuum light speed of the original Maxwell equations and $c_h \in \mathbb{R}^+$ is the speed associated with the two acoustic subsystems. Note that in the above system we have replaced the electric field of the original Maxwell equations by a rescaled electric field
$c_0 \E$, in order to make the symmetric hyperbolicity more evident. Throughout this paper we assume $c_0$ and $c_h$ to be constant in space and time. The above system has an interesting mixed combination of curl-curl and div-grad operators that goes beyond the standard Maxwell equations or those for linear acoustics. The same structure is present in the tetrad formulation of the Einstein field equations of Buchman \textit{et al.} \cite{Buchman1,Buchman2} for certain gauge choices, which makes the above system an interesting mathematical prototype system. In the following we show that the above system admits an extra conservation law for the total energy density 
\begin{equation}
	\mathcal{E} = \halb \left( \E^2 + \B^2 \right) + \halb \left( \psi^2 + \phi^2 \right).  \label{eqn.totalenergy} 
\end{equation}
We follow the pioneering work of Godunov \cite{God1961} and Romenski \cite{Rom1998} and introduce the state vector $\q = \left(  \B, \phi,\E, \psi \right)$, as well as the associated vector of thermodynamic \textit{dual} variables, the so-called \textit{main field} of Ruggeri \textit{et al.} \cite{Ruggeri81}, defined as $\p = \partial \mathcal{E} / \partial \q$. Furthermore, we introduce the so-called \textit{generating potential} $L$, which is the Legendre transform of $\en$ and is defined as
\begin{equation}
	L = \p \cdot \q - \en. 
\end{equation}
For the simple quadratic energy potential \eqref{eqn.totalenergy} we immediately obtain $\p = \q = \left( \B, \phi,\E, \psi \right)$ and $L=\en$.  
Multiplying \eqref{eqn.B}-\eqref{eqn.psi} with the dual variables $\p$ and summing up
\begin{equation}\label{eqn.sum}
	\B\cdot\eqref{eqn.B}+\phi\cdot\eqref{eqn.phi}+\E\cdot\eqref{eqn.E}+\psi\cdot\eqref{eqn.psi}
\end{equation}
yields, after simple calculations, the following total energy conservation law
\begin{equation}
	\partial_t \en + \nabla \cdot \left( c_0 \E \times \B + c_h (\psi \E + \phi \B) \right) = 0\,, 
	\label{eqn.energy} 
\end{equation}
where $\E \times \B$ is the Poynting vector. It is obvious that for initial data that are compatible with the original vacuum Maxwell equations, i.e. that satisfy $\nabla \cdot \B = 0$, $\nabla \cdot \E = 0$ and $\phi=const$, $\psi=const$ at the initial time $t=0$ the divergence of the magnetic and of the electric field remain zero for all times. 

\subsection{Symmetric hyperbolicity}

The system \eqref{eqn.B}-\eqref{eqn.psi} is \textit{symmetric hyperbolic} in the sense of Friedrichs \cite{FriedrichsSymm}, since it can be written in the form 
\begin{equation} 
	\M \partial_t \p + \H_k \partial_k \p = 0,
\end{equation} 
with $\M = L_{\p\p} = \I$, where $\I$ is the identity matrix and the matrices $\H_k$ are given by 
\begin{equation}\label{eqn.H.1}
 \H_1 = \left( \begin{array}{cccccccc} 
 	0 & 0 & 0 & c_h & 0 & 0 & 0 &  0 \\ 
	0 & 0 & 0 & 0 & 0 & 0 & -c_0 & 0 \\
	0 & 0 & 0 & 0 & 0 &  c_0 & 0 & 0 \\
	 c_h & 0 & 0 & 0 & 0 & 0 & 0 & 0 \\ 
 	0 & 0 & 0 & 0 & 0 & 0 & 0 & c_h \\ 
 	0 & 0 & c_0 & 0 & 0 & 0 &  0 & 0 \\
 	0 & -c_0 & 0 & 0 & 0 & 0 & 0 & 0 \\
 	0 & 0 & 0 & 0 &  c_h & 0 & 0 & 0 \\ 
 \end{array} \right),   	 
\end{equation}
\begin{equation}\label{eqn.H.2}
 \H_2 = \left( \begin{array}{cccccccc} 
	0 & 0 &  0 & 0 & 0 & 0 & c_0 & 0 \\ 
	0 & 0 & 0 & c_h & 0 & 0 & 0 &  0 \\
	0 & 0 & 0 & 0 & -c_0 & 0 & 0 & 0 \\
	0 & c_h & 0 & 0 & 0 &  0 & 0 & 0 \\ 
	0 & 0 & -c_0 & 0 & 0 & 0 & 0 & 0 \\ 
	0 & 0 & 0 &  0 & 0 & 0 & 0 & c_h \\
	c_0 & 0 & 0 & 0 &  0 & 0 & 0 & 0 \\
	0 & 0 & 0 & 0 & 0 & c_h & 0 & 0 \\ 
\end{array} \right),   	 
\end{equation}
\begin{equation}\label{eqn.H.3}
	\H_3 = \left( \begin{array}{cccccccc} 
		0 & 0 & 0 & 0 & 0 & -c_0 & 0 & 0 \\ 
		0 & 0 & 0 & 0 & c_0 & 0 & 0 & 0 \\
		0 & 0 & 0 & c_h & 0 & 0 & 0 & 0 \\
		0 & 0 & c_h & 0 & 0 & 0 &  0 & 0 \\ 
		0 & c_0 & 0 & 0 & 0 &  0 & 0 & 0 \\ 
		-c_0 & 0 & 0 & 0 & 0 & 0 & 0 & 0 \\
		0 & 0 & 0 &  0 & 0 & 0 & 0 & c_h \\
		0 & 0 & 0 & 0 & 0 & 0 & c_h & 0 \\ 
	\end{array} \right). 
\end{equation}
The matrices $\H_i$ are obviously symmetric and the matrix $\M=\I$ is clearly symmetric and strictly positive definite, which is an immediate consequence of the strict convexity of the quadratic total energy potential. We therefore conclude that system \eqref{eqn.B}-\eqref{eqn.psi} is symmetric hyperbolic in the sense of Friedrichs \cite{FriedrichsSymm}. For the sake of completeness, we also give the eigenvalues and eigenvectors of the system:
The eigenvalues of all matrices $\H_i$ are 
\begin{equation}
	\lambda_{1,2} = -c_h, \qquad 
	\lambda_{3,4} = -c_0, \qquad 
	\lambda_{5,6} = +c_0, \qquad 
	\lambda_{7,8} = +c_h.
\end{equation} 	
It is important to realize that the eigenvalues associated to $c_h$ do not correspond to the propagation of any physical signal, which allows us to analyze the case when $c_h>c_0$. 
The associated matrix of right eigenvectors for $\H_1$ reads 
\begin{equation}
	 \mathbf{R} = \left( \begin{array}{cccccccc} 
	 	-1 &  0 &  0 &  0 &  0 &  0 & 0 & 1 \\ 
	 	 0 &  0 &  1 &  0 &  0 & -1 & 0 & 0 \\
	 	 0 &  0 &  0 & -1 &  1 &  0 & 0 & 0 \\
	 	 1 &  0 &  0 &  0 &  0 &  0 & 0 & 1 \\ 
	 	 0 & -1 &  0 &  0 &  0 &  0 & 1 & 0 \\ 
	 	 0 &  0 &  0 &  1 &  1 &  0 & 0 & 0 \\
	 	 0 &  0 &  1 &  0 &  0 &  1 & 0 & 0 \\
	 	 0 &  1 &  0 &  0 &  0 &  0 & 1 & 0 \\ 
	 \end{array} \right). 
\end{equation}
According to \cite{Kato1975,Serre2007} the symmetric hyperbolicity of the system guarantees well-posedness of the initial value problem, at least for short times. 

\subsection{Formal asymptotic limit for $c_h \to \infty$} 

In this section we carry out a formal asymptotic analysis of the Maxwell-Munz system in the case of $c_h \to \infty$ for constant $c_0$, i.e. when the dimensionless parameter $\epsilon = c_0 / c_h \to 0$. In this paper we follow the procedure outlined in \cite{klein,roller2} for the compressible Navier-Stokes equations, while a mathematically rigorous analysis of the incompressible limit of the compressible Euler, Navier-Stokes and MHD equations was provided in \cite{KlaMaj,KlaMaj82}. Similar to  \cite{KlaMaj,KlaMaj82} we assume well-prepared initial data. In particular, we require that $\nabla \cdot \B(\x,0) = \mathcal{O}(\epsilon^2)$ and $\nabla \cdot \E(\x,0) = \mathcal{O}(\epsilon^2)$ hold, as well as $\phi(\x,0) = \phi_0 = const$ and $\psi(\x,0) = \psi_0 = const$. Expanding all variables in a power series of $\epsilon$, the so-called Hilbert expansion, yields 
\begin{eqnarray}
	\B &=& \B_0 + \epsilon \B_1 + \epsilon^2 \B_2 + \mathcal{O}(\epsilon^3), 
	\label{eqn.B.series} \\ 
	\phi &=& \phi_0 + \epsilon \phi_1 + \epsilon^2 \phi_2 + \mathcal{O}(\epsilon^3), 
	\label{eqn.phi.series} \\ 
	\E &=& \E_0 + \epsilon \E_1 + \epsilon^2 \E_2 + \mathcal{O}(\epsilon^3), 
	\label{eqn.E.series} \\ 
	\psi &=& \psi_0 + \epsilon \psi_1 + \epsilon^2 \psi_2 + \mathcal{O}(\epsilon^3). 
	\label{eqn.psi.series} 
\end{eqnarray}
Inserting the above series expansions into all equations \eqref{eqn.B}-\eqref{eqn.psi} divided by $c_0$ and collecting terms with equal powers in $\epsilon = c_0/c_h$ leads to  
\begin{eqnarray}
	\epsilon^{-1} \nabla \phi_0 + \epsilon^0 \left( c_0^{-1} \partial_t \B_0 + \nabla \times \E_0 + \nabla \phi_1 \right) + \mathcal{O}(\epsilon) &=& 0, \label{eqn.B.eps} \\ 
	\epsilon^{-1} \nabla \cdot \B_0 + \epsilon^0 \left( c_0^{-1} \partial_t \phi_0  + \nabla \cdot \B_1 \right)  + \mathcal{O}(\epsilon) &=& 0. \label{eqn.phi.eps}   \\	
	\epsilon^{-1} \nabla \psi_0 + \epsilon^0 \left( c_0^{-1} \partial_t \E_0 - \nabla \times \B_0 + \nabla \psi_1 \right) + \mathcal{O}(\epsilon) &=& 0, \label{eqn.E.eps} \\ 
	\epsilon^{-1} \nabla \cdot \E_0 + \epsilon^0 \left( c_0^{-1} \partial_t \psi_0  + \nabla \cdot \E_1 \right)  + \mathcal{O}(\epsilon) &=& 0. \label{eqn.psi.eps}  	
\end{eqnarray}
Note that all quantities $\B$, $\phi$, $\E$ and $\psi$ must be \textit{bounded} due to the energy conservation principle \eqref{eqn.energy}. The equations \eqref{eqn.phi.eps} and \eqref{eqn.psi.eps} immediately provide the result that the leading order magnetic field and the leading order electric field must become divergence-free for $\epsilon \to 0$, i.e. 
\begin{equation}
	\nabla \cdot \B_0 = 0, \qquad \textnormal{and} \qquad 
	\nabla \cdot \E_0 = 0. 
	\label{eqn.div.O0} 
\end{equation}
Furthermore, from \eqref{eqn.B.eps} and \eqref{eqn.E.eps} we obtain that 
\begin{equation}
	\nabla \phi_0 = 0, \qquad \textnormal{and} \qquad 
	\nabla \psi_0 = 0,
\end{equation}
which means that for $\epsilon \to 0$ the leading order terms of the two scalar fields become constant in space and thus depend only on time, i.e. $\phi_0 = \phi_0(t)$ and $\psi_0 = \psi_0(t)$. Taking the terms of order $\epsilon^0=1$ in \eqref{eqn.phi.eps} and \eqref{eqn.psi.eps} yields 
\begin{equation}
	c_0^{-1} \partial_t \phi_0 + \nabla \cdot \B_1 = 0, 
	\qquad 
	\textnormal{and}
	\qquad 
	c_0^{-1} \partial_t \psi_0 + \nabla \cdot \E_1 = 0.
\end{equation} 
Since $\phi_0$ and $\psi_0$ are constant in space and only functions in time, application of the Gauss theorem provides the time rate of change of both $\phi_0$ and $\psi_0$ as 
\begin{equation}
	\frac{d \phi_0}{dt} =  - c_0 \int \limits_{\partial \Omega} \B_1 \cdot \mathbf{n} dS, 
	\qquad 
	\textnormal{and}
	\qquad 
	\frac{d \psi_0}{dt} =  - c_0 \int \limits_{\partial \Omega} \E_1 \cdot \mathbf{n} dS. 
	\label{eqn.gauss} 
\end{equation} 
Assuming that the following \textit{boundary conditions} hold for both, the magnetic and the electric field 
\begin{equation}
	\B \cdot \mathbf{n} = 0, 	 \qquad \forall \x \in \partial \Omega, \qquad  \qquad 
	\E \cdot \mathbf{n} = 0, 	 \qquad \forall \x \in \partial \Omega 
	\label{eqn.BC} 
\end{equation}
we obtain from \eqref{eqn.gauss} that $\frac{d \phi_0}{dt} = \frac{d \psi_0}{dt} = 0$. 
Hence, from the terms of order $\epsilon^0$ of \eqref{eqn.phi.eps} and \eqref{eqn.psi.eps} we find that
with the boundary conditions \eqref{eqn.BC} we also have 
\begin{equation}
	\nabla \cdot \B_1 = 0, \qquad \textnormal{and} \qquad \nabla \cdot \E_1 = 0. 
	\label{eqn.div.O1} 
\end{equation}
Thus, from 
\eqref{eqn.B.series} and \eqref{eqn.E.series} 
we therefore can conclude that for $\epsilon \to 0$ the divergence errors of the magnetic and of the electric field scale quadratically in $\epsilon$, i.e. 
\begin{equation}
	 \nabla \cdot \B = \mathcal{O}(\epsilon^2), \qquad \textnormal{and} \qquad 
	 \nabla \cdot \E = \mathcal{O}(\epsilon^2).
\end{equation} 
Note that this is not a mathematically rigorous proof, but just a formal asymptotic analysis. We also emphasize the importance of the assumption of well-prepared initial data.

\subsection{Underlying variational principle} 

In order to derive \eqref{eqn.B}-\eqref{eqn.psi} from an underlying variational principle we introduce a vector potential $\A$ and a scalar potential $Z$. We then make the following \textit{definitions}: 
\begin{equation}
	\E = - \frac{1}{c_0}\partial_t \A, \qquad \phi = - \frac{1}{c_h}\partial_t Z, \label{eqn.dynamic} 
\end{equation}
and 
\begin{equation}
	\B =  \nabla \times \A +  \nabla Z, \qquad \psi = \frac{c_h}{c_0} \nabla \cdot \A. \label{eqn.spatial}
\end{equation}
We also introduce the additional quantity 
\begin{equation}
	\Y = \nabla Z.
\end{equation}
Since $c_0$ and $c_h$ are constant, an \textit{immediate consequence} of the above definitions combined with classical vector calculus identities are the following relations: 
\begin{equation}
	\nabla \cdot \B =  \nabla^2 Z =  \nabla \cdot \Y, \label{eqn.pressurecorrection} 	
\end{equation}
and
\begin{equation}
	\nabla \times \B = \nabla \times \nabla \times \A, \label{eqn.rotB} 	
\end{equation}
the first one of which is similar in structure to the pressure-correction equation in the context of numerical methods for the incompressible Navier-Stokes equations \cite{HW65,patankar,PS72}, while the second one is classical in electrodynamics. Furthermore, applying the time derivative to \eqref{eqn.spatial} and using the definitions \eqref{eqn.dynamic} we immediately obtain the time evolution equation for the magnetic field as 
\begin{equation}
	\partial_t \B -  \nabla \times \partial_t \A -  \nabla \partial_t Z = 0, \qquad 
	\textnormal{hence} \qquad  
	\partial_t \B + c_0 \nabla\times \E + c_h \nabla \phi = 0,
	\label{eqn.B2} 
\end{equation}
and the time evolution equation for the scalar field $\psi$ as 
\begin{equation}
	\partial_t \psi - \frac{c_h}{c_0} \nabla \cdot \partial_t \A = 0, \qquad 
	\textnormal{hence} \qquad  
	\partial_t \psi + c_h \nabla \cdot \E = 0. 
	\label{eqn.psi2} 
\end{equation}
It is obvious that \eqref{eqn.B2} and \eqref{eqn.psi2} are identical with \eqref{eqn.B} and \eqref{eqn.psi}, respectively. 
A further immediate consequence of the above definitions is the PDE
\begin{equation}
	\partial_t \Y + \nabla \phi = 0, \label{eqn.Y} 
\end{equation}
but for which we have so far found no use. 
In order to obtain the missing governing PDE for $\E$ and $\phi$, we now postulate the following Lagrangian density in terms of $\A$ and $Z$: 
\begin{equation}
	\Lambda = \halb \left[ \left(\frac{1}{c_0} \partial_t \A \right)^2 -  \left( \nabla \times \A \right)^2 -  \left(\frac{c_h}{c_0} \nabla \cdot \A \right)^2 + \left(\frac{1}{c_h}\partial_t Z \right)^2 -  \nabla Z \cdot \nabla Z \right], 
	\label{eqn.Lagrangian}
\end{equation}
or, in index notation, 
\begin{equation}
	\Lambda = \halb \left(  \frac{1}{c_0^2}\partial_t A_i \partial_t A_i -  \epsilon_{lpq} \partial_p A_q  \epsilon_{lrs} \partial_r A_s - \frac{c_h^2}{c_0^2} \partial_r A_r \delta_{pq} \partial_p A_q + \left( \frac{1}{c_h}\partial_t Z \right)^2 -  \partial_i Z \partial_i Z \right)\,,
	\label{eqn.Lagrangian}
\end{equation}
where $\delta_{ik}$ is the usual Kronecker delta, while $\epsilon_{ijk}$ is the three-dimensional Levi-Civita tensor.
%
We now state the following \textit{variational principle}: 
Find $\A$ and $Z$ so that the action 
\begin{equation}
 	S[A_k,Z] = \int_{\Omega} \Lambda \, \mathrm{d}\x \mathrm{d} t
 	\label{eqn.action} 
\end{equation}
is minimized. 
\noindent 
The first Euler-Lagrange equation related to \eqref{eqn.action} and the scalar potential $Z$ reads ($Z_{,t}:= \partial_t Z $, $Z_{,k}:=\partial_k Z$): 
\begin{equation}
	-\frac{\partial}{\partial t} \frac{\partial \Lambda}{\partial Z_{,t}} - \frac{\partial}{\partial x_k} \frac{\partial \Lambda}{\partial Z_{,k}} + \frac{\partial \Lambda}{\partial Z} = 0, 	 
\end{equation}
which immediately becomes 
\begin{equation}
	-\frac{1}{c_h^2}\frac{\partial Z_{,t}}{\partial t} +  \delta_{ik} \frac{\partial Z_{,i}}{\partial x_k} = 0\,. 	 
	\label{eqn.el.phi} 
\end{equation}
\begin{equation}
	\frac{1}{c_h}\partial_t \phi +  \nabla^2 Z = 0,
\end{equation}
which thanks to \eqref{eqn.pressurecorrection} becomes the final evolution equation of $\phi$, i.e. \eqref{eqn.phi}: 
\begin{equation}\label{eqn.phi2}
	\partial_t \phi + c_h \nabla \cdot \B = 0.
\end{equation}
Likewise, the Euler-Lagrange equation related to \eqref{eqn.action} and the vector potential $\A$ reads ($A_{i,t}:=\partial_t A_i $, $A_{i,k}:=\partial_k A_i $):  
\begin{equation}
	-\frac{\partial}{\partial t} \frac{\partial \Lambda}{\partial A_{i,t}} - \frac{\partial}{\partial x_k} \frac{\partial \Lambda}{\partial A_{i,k}} + \frac{\partial \Lambda}{\partial A_i} = 0, 	 
\end{equation}
which translates into 
\begin{equation}
	 -\frac{1}{c_0^2}\frac{\partial A_{i,t}}{\partial t} +   \frac{\partial}{\partial x_k} \epsilon_{lpq} A_{q,p} \epsilon_{lrs} \delta_{kr} \delta_{is}  + \frac{c_h^2}{c_0^2} \frac{\partial}{\partial x_k} A_{r,r} \delta_{pq} \delta_{pk} \delta_{qi}  = 0, 
\end{equation}
or, equivalently, 
\begin{equation}
	-\frac{1}{c_0^2}\frac{\partial A_{i,t}}{\partial t} + \frac{\partial}{\partial x_k}  \epsilon_{lpq} A_{q,p} \epsilon_{lki}  + \frac{c_h^2}{c_0^2}  \frac{\partial}{\partial x_k} A_{r,r} \delta_{ki} = 0. 
	\label{eqn.el.A}
\end{equation}
In standard vector notation and using the identities \eqref{eqn.dynamic} and \eqref{eqn.rotB}, the above equation finally becomes 
\begin{equation}\label{eqn.E2}
  \partial_t \E - c_0 \nabla \times \B + c_h \nabla \psi = 0, 
\end{equation}
which is identical with \eqref{eqn.E}. We therefore have established the underlying variational principle from which the Maxwell-Munz system can be derived. We stress that the procedure shown in this paper is different from the one introduced in \cite{GLMLagrangian}. 

It is interesting to note that the PDE for the magnetic field $\B$ \eqref{eqn.B} and for the scalar field $\psi$ are immediate consequences of the definitions, while the PDE for the electric field $\E$ and for the scalar quantity $\phi$ are the Euler-Lagrange equations associated with the variational principle \eqref{eqn.Lagrangian}. As usual in the framework of SHTC equations \cite{Rom1998}, the equations come in pairs. One being an Euler-Lagrange equation, the other one being a mere consequence of the definitions.  

\subsection{Variational formulation for arbitrary nonlinear Lagrangian}

Let us show that the variational derivation described in the previous section
can be generalized to an arbitrary Lagrangian, which might be important for
establishing connections between the Maxwell-GLM system and other nonlinear
systems with the similar structure, e.g. \cite{Buchman1,Buchman2}.
Let us consider an action integral for a vector $A_k$ and a scalar potential $Z$:  
\begin{equation}
	S[A_k,Z] = \int \Lambda(A_k,Z; \pd_t A_k, \pd_t Z, \pd_k A_j, \pd_k Z) \mathrm{d}\x \mathrm{d}t.
\end{equation}
We are interested in reparametrizing the Lagrangian $\Lambda$ in the following variables
\begin{alignat}{2}
	&E_k = -\pd_t A_k =-A_{k,t}, 						&& \varphi =-\pd_t Z=-Z_{,t}\, ,\label{eqn.definitions1}\\
	&B_k = \epsilon_{kij} \pd_i A_j + \pd_k Z = \epsilon_{kij} A_{j,i} + Z_{,k}, \qquad 	&& \psi = \pd_k A_k = A_{k,k}\, ,\label{eqn.definitions2}
\end{alignat}
i.e.
\begin{equation}
	\mathcal{L}(E_k,B_k,\varphi,\psi) = \Lambda(A_k,Z; \pd_t A_k, \pd_t Z, \pd_k A_j, \pd_k Z).
\end{equation}
Note that the definitions \eqref{eqn.definitions1} and \eqref{eqn.definitions2} 
differ from \eqref{eqn.dynamic} and \eqref{eqn.spatial} only by the scaling
factors, which, without the loss of generality, are absorbed into the fields $E_k$, $\varphi$, $B_k$, and $\psi$
in this section.
The relations between the derivatives of the parametrizations $\mathcal{L}$ and $\Lambda$ read
 \begin{alignat}{2}
	\frac{\pd \Lambda}{\pd A_{k,t}} & =-\frac{\pd \mathcal{L}}{\pd E_k},
	\qquad
	&\frac{\pd\Lambda}{\pd Z_t} 		 = -\frac{\pd \mathcal{L}}{\pd \varphi},\\
	\frac{\pd\Lambda}{\pd A_{k,j}}  & = \epsilon_{jki} \frac{\pd \mathcal{L}}{\pd B_i} + \delta_{kj} \frac{\pd \mathcal{L}}{\pd \psi},
	\qquad
	&\frac{\pd\Lambda}{\pd Z_k} 		 = \frac{\pd \mathcal{L}}{\pd B_k}.
\end{alignat}
Therefore, assuming $\pd \Lambda/\pd A_k = 0$ and $\pd \Lambda/ \pd Z = 0$, the
variation with respect to $A_k$ and $Z$ gives the Euler-Lagrange equations (for
brevity, we adopt the  notations $\mathcal{L}_{E_k} = \frac{\pd \mathcal{L}}{\pd
E_k}$, etc.):
\begin{align}
	\frac{\pd \mathcal{L}_{E_k}}{\pd t} & + \epsilon_{kij} \pd_i \mathcal{L}_{B_j} - \pd_k \mathcal{L}_\psi = 0, \label{eq:EL.E}\\
	\frac{\pd \mathcal{L}_\varphi}{\pd t} & -\pd_k \mathcal{L}_{B_k} = 0,\label{eq:EL.phi}
\end{align}
that must be accompanied by the integrability conditions (consequences of the definitions \eqref{eqn.definitions1}, \eqref{eqn.definitions2}):
\begin{align}
	\frac{\pd B_k}{\pd t} & + \epsilon_{kij} \pd_i E_j + \pd_k \varphi = 0,\label{eq:integrability.B}\\	
	\frac{\pd \psi}{\pd t} & + \pd_k E_k = 0.
	\label{eq:integrability.psi}
\end{align}
After applying a partial Legendre transform to
$\mathcal{L}(E_k,B_k,\varphi,\psi)$ and introducing the new state variables 
\begin{equation}
	D_k := \mathcal{L}_{E_k}, \qquad \phi := \mathcal{L}_\varphi,
	\qquad
	\en(D_k,B_k,\phi,\psi) := E_k \mathcal{L}_{E_k} + \varphi \mathcal{L}_\varphi - \mathcal{L},
\end{equation}
with the following relations between the old and new parametrization
\begin{equation}
	\en_{B_k} = -\mathcal{L}_{B_k}, 
	\qquad 
	\en_{\psi} = -\mathcal{L}_{\psi},
	\qquad
	E_k = \en_{D_k},
	\qquad
	\varphi = \en_{\phi},
\end{equation}
equations \eqref{eq:EL.E}-\eqref{eq:EL.phi} and \eqref{eq:integrability.B}-\eqref{eq:integrability.psi} become:
\begin{align}
	\frac{\pd B_k}{\pd t} + \epsilon_{kij} \pd_i \en_{D_j} + \pd_k \en_\phi &= 0,\label{eq:SHTC.B}\\
	\frac{\pd \phi}{\pd t} + \pd_k \en_{B_k} &= 0, \\
	\frac{\pd D_k}{\pd t} - \epsilon_{kij} \pd_i \en_{B_j} + \pd_k \en_\psi  &= 0,\\
	\frac{\pd \psi}{\pd t} + \pd_k \en_{D_k} &= 0. \label{eq:SHTC.psi}
\end{align}
This nonlinear system can be immediately symmetrized in the dual variables
$\p=(\en_{B_k},\en_{\phi},\en_{D_k},\en_{\psi})$  and in the new potential (full Legendre transform)
\begin{equation}
	L(\p) = D_k \en_{D_k} + B_k \en_{B_k} + \phi \en_{\phi} + \psi \en_{\psi} - \en
\end{equation}
because it has exactly the same symmetric structure of the spatial part as in
\eqref{eqn.H.1}-\eqref{eqn.H.3}, and therefore, it can be called a symmetric
hyperbolic system in the case the potential $\en$ is convex, e.g. see
\cite{Rom1998}. Moreover, note that system
\eqref{eq:SHTC.B}--\eqref{eq:SHTC.psi} has an extra conservation law (it can be
trivially obtained by multiplying each equation by the corresponding dual
variable and summing them up, exactly as in \eqref{eqn.sum}):
\begin{equation}\label{eqn.energy.nonlinear}
	\frac{\pd \en}{\pd t} + \pd_k \left( \epsilon_{kij} \en_{D_i} \en_{B_j} + \en_\psi \en_{D_k} + \en_\phi \en_{B_k} \right) = 0.
\end{equation}
It therefore can be categorized as a symmetric hyperbolic thermodynamically compatible system \cite{Rom1998}.
Let us show how this system can be reduced to the Maxwell-GLM system. Consider
the following energy potential
\begin{equation}\label{eqn.energy2}
	\en = \frac{1}{2} \left( \frac{1}{\varepsilon} \Vert \D \Vert^2 + \frac{1}{\mu} \Vert \B \Vert^2\right)
	+
	\frac{1}{2}\left( \alpha^2 \phi^2 + \beta^2 \psi^2 \right),
	\qquad 
	c_0^2 = \frac{1}{\varepsilon \mu}.
\end{equation}
and assume that 
\begin{equation}
	\alpha^2 = \mu c_h^2, \qquad \beta^2 = \varepsilon c_h^2,
\end{equation}
and let us also rescale the fields as
\begin{equation}
	\hat{B}_k := B_k,
	\qquad
	\hat{E}_k := \frac{1}{c_0} E_k,
	\qquad
	\hat{\phi} := c_h \mu \phi,
	\qquad
	\hat{\psi} := \frac{c_h}{c_0}\psi.
\end{equation}
After this, system \eqref{eq:SHTC.B}-\eqref{eq:SHTC.psi} becomes the Maxwell-GLM system:
\begin{align}
	\frac{\pd \hat{B}_k}{\pd t} + c_0 \epsilon_{kij} \pd_i \hat{E}_j + c_h \pd_k \hat{\phi} &= 0,\\
	\frac{\pd \hat{\phi}}{\pd t} + c_h \pd_k \hat{B}_k &= 0,\\	
	\frac{\pd \hat{E}_k}{\pd t} - c_0\epsilon_{kij} \pd_i \hat{B}_j + c_h \pd_k \hat{\psi}  &= 0,\\
	\frac{\pd \hat{\psi}}{\pd t} + c_h \pd_k \hat{E}_k &= 0.
\end{align}
%

\subsection{Hamiltonian nature of the Maxwell-GLM system}

An intimate connection between the Lagrangian and Hamiltonian formulations of
mechanics is well-established, see e.g. \cite{Hamill2013}. It is also known that
in addition to the Lagrangian formulation, the Maxwell equations possess a
Hamiltonian formulation, i.e. their time evolution PDEs can be generated by a
Poisson bracket, see e.g. \cite[Sec.3.10.1]{PKG_Book2018} or
\cite{SHTC-GENERIC-CMAT}. The Hamiltonian formulation of mechanics and
electromagnetism is a powerful theoretical tool as it is well-known from the
history of theoretical physics, as well as it is important for designing
numerical methods, e.g. see \cite{Morrison2017}. It is therefore interesting to
investigate if the Maxwell system with the GLM divergence cleaning still
possesses a Hamiltonian formulation. 

The answer to this question is positive. Indeed, using the ``reverse engineering''
approach (i.e. starting from the PDEs), see \cite[Sec.\,A.5]{PKG_Book2018} or
\cite[Eq.(68)-(69)]{SHTC-GENERIC-CMAT}, the following Poisson bracket
for system \eqref{eq:SHTC.B}-\eqref{eq:SHTC.psi}, and subsequently for the
Maxwell-Munz system \eqref{eqn.B}-\eqref{eqn.psi}, can be obtained
\begin{multline}
	\left\{ \mathcal{A},\mathcal{B} \right\}^{(B_k,D_k,\psi,\phi)} 
	= \int \Big(
		\left[ \mathcal{A}_{D_k} \left( \varepsilon_{kij}\pd_i \mathcal{B}_{B_j} - \pd_k \mathcal{B}_\psi\right) - \left( \varepsilon_{kij}\pd_i \mathcal{A}_{B_j} - \pd_k \mathcal{A}_\psi\right) \mathcal{B}_{D_k} \right] 
	+ \\
	\left[ \mathcal{A}_{\phi} \left( -\pd_k \mathcal{B}_{B_k} \right) - \left(-\pd_k \mathcal{A}_{B_{k}} \right) \mathcal{B}_{\phi} \right]
	\Big)
	\mathrm{d}\x,
	\label{eqn.PB}
\end{multline}
where $\mathcal{A}$ and $\mathcal{B}$ are arbitrary sufficiently smooth
functionals of the variables $(B_k,D_k,\psi,\phi)$ and possibly of their first
spatial gradients, and $\mathcal{A}_{B_k}$, $\mathcal{B}_{B_k}$, etc. are the
functional derivatives with respect to $B_k$, $D_k$, $\psi$, $\phi$,
respectively, i.e. $\mathcal{A}_{B_k} = \frac{\delta \mathcal{A}}{\delta B_k}$,
etc. However, as it follows from the previous discussion, we assume that the
functionals do not depend on the gradients of the fields, and therefore their
functional derivatives are equal to the partial derivatives, i.e. $\mathcal{A}_{B_k}
= \frac{\pd \mathcal{A}}{\pd B_k}$, etc.

The Poisson bracket \eqref{eqn.PB}, however, is not canonical, and hence the Jacobi
identity must be checked for it. The Jacobi identity was checked with the help
of the symbolic computation software Mathematica, and in particular with the
package \cite{krogerHutter2010}, and it was found that the Jacobi identity holds
for the bracket \eqref{eqn.PB}. 

To see how equations \eqref{eq:SHTC.B}-\eqref{eq:SHTC.psi} unfold from the
Poisson bracket $\left\{ \mathcal{A},\mathcal{B}
\right\}^{(B_k,D_k,\psi,\phi)}$, one can apply integration by parts in
\eqref{eqn.PB}, and rewrite the bracket as
\begin{align}
	\left\{ \mathcal{A},\mathcal{B} \right\}^{(B_k,D_k,\psi,\phi)} = \int 
	\Big(
	&\mathcal{A}_{B_k} \left( -\varepsilon_{kij} \pd_i \mathcal{B}_{D_j} - \pd_k \mathcal{B}_\phi \right) +
	\mathcal{A}_{\phi} \left( -\pd_k \mathcal{B}_\phi \right) + \nonumber\\
	&\mathcal{A}_{D_k} \left( \varepsilon_{kij} \pd_i \mathcal{B}_{B_j} - \pd_k \mathcal{B}_\psi \right) +
	\mathcal{A}_{\psi} \left( -\pd_k\mathcal{B}_{D_k} \right) 
	\Big)\mathrm{d} \x.
	\label{eqn.PB.time.evol}
\end{align}

On the other hand, in the Hamiltonian formalism, the reversible time evolution
of an arbitrary functional $\mathcal{A}(B_k,D_k,\psi,\phi) = \int A
\mathrm{d}\x$ is given as 
\begin{equation}\label{eqn.time.evol.Ham}
	\frac{\pd \mathcal{A}}{\pd t} = \left\{ \mathcal{A},\en \right\}^{(B_k,D_k,\psi,\phi)},
\end{equation}
with $\en$ being the energy of the system.
Therefore, after noting that the left
hand-side of \eqref{eqn.time.evol.Ham} can be also formally written as
\begin{equation}\label{eqn.time.evol.sum}
	\frac{\pd \mathcal{A}}{\pd t} = \int 
	\left(	
	\frac{\delta \mathcal{A}}{\delta B_k} \frac{\pd B_k}{\pd t} + 
	\frac{\delta \mathcal{A}}{\delta \phi} \frac{\pd \phi}{\pd t} + 
	\frac{\delta \mathcal{A}}{\delta D_k} \frac{\pd D_k}{\pd t} + 
	\frac{\delta \mathcal{A}}{\delta \psi} \frac{\pd \psi}{\pd t} 
	\right)
	\mathrm{d} \x\ ,
\end{equation}
and due to the arbitrariness of the functional $\mathcal{A}$, one can directly
read the expressions for the time evolutions of the corresponding variables in
the parentheses in \eqref{eqn.PB.time.evol}, e.g. $\frac{\pd B_k}{\pd t} =
-\varepsilon_{kij} \pd_i \en_{D_j} - \pd_k \en_\phi$, which is exactly
\eqref{eq:SHTC.B}, etc.
Recall that the energy conservation law \eqref{eqn.energy.nonlinear} can be
trivially obtained in the Hamiltonian formalism as
\begin{equation}
	\frac{\pd \en}{\pd t} = \left\{ \en,\en \right\}^{(B_k,D_k,\psi,\phi)} = 0,
\end{equation}
and \eqref{eqn.time.evol.sum} is another manifestation of how the energy
conservation laws \eqref{eqn.energy} and \eqref{eqn.energy.nonlinear} were
derived.
Finally, we remark that if one takes the pairs $(A_k,D_k)$ and $(Z,\phi)$ as the
canonical variables, and hence admits that the functionals may depend on their spatial gradients then there exists a canonical Poisson bracket $\left\{ \mathcal{A},\mathcal{B} \right\}^{(A_k,D_k,Z,\phi)}$
\begin{equation}\label{eqn:P.bracket}
	\left\{ \mathcal{A},\mathcal{B} \right\}^{(A_k,D_k,Z,\phi)} = \int 
	\Big(
	\left( \mathcal{A}_{D_k} \mathcal{B}_{A_k} - \mathcal{A}_{A_k} \mathcal{B}_{D_k} \right) 
	+
	\left( \mathcal{A}_{\phi} \mathcal{B}_{Z} - \mathcal{A}_{Z} \mathcal{B}_{\phi} \right)
	\Big)
	\mathrm{d}\x,
\end{equation}
where (keeping in mind definitions \eqref{eqn.definitions1},
\eqref{eqn.definitions2} and assuming the functionals do not depend explicitly
on $A_k$ and $Z$), the functional derivatives w.r.t. the variables $A_k$ and $Z$
read
\begin{align}
	\mathcal{A}_{A_k} &= \frac{\delta \mathcal{A}}{\delta A_k} = \frac{\partial \mathcal{A}}{\partial A_k} - \pd_i \left( \frac{\partial \mathcal{A}}{\partial A_{k,i}} \right) 
	= \varepsilon_{kij} \pd_i \mathcal{A}_{B_j} - \pd_k \mathcal{A}_{\psi}, \\
	\mathcal{A}_{Z} &= \frac{\delta \mathcal{A}}{\delta Z} = \frac{\partial \mathcal{A}}{\partial Z} - \pd_i \left( \frac{\partial \mathcal{A}}{\partial Z_i} \right) = -\pd_i \mathcal{A}_{B_i}.
\end{align}
Using these formulas, one can transform the bracket
$\{\mathcal{A},\mathcal{B}\}^{(A_k,D_k,Z,\phi)}$ to the bracket
$\{\mathcal{A},\mathcal{B}\}^{(B_k,D_k,\psi,\phi)}$. This transformation,
however, is not canonical. Of course, the time evolution equations corresponding
to $\{\mathcal{A},\mathcal{B}\}^{(A_k,D_k,Z,\phi)}$ are second-order PDEs on
$(A_k,D_k,Z,\phi)$, which is not within our interests.

\section{Extension  to special relativity}
\label{sec.sr}

In this section we use the Einstein summation convention of repeated indices.
Greek indices run from $0$ to $3$, Latin indices run from $1$ to $3$. 
Moreover, we set $c_0 = c_h = 1$ and make use of the Lorentz-Heaviside notation, 
which is equivalent to the Gauss notation where all $4\pi$ factors in the original Maxwell equations disappear.
In the following we present the extension to special relativity of the equations discussed so far. 
The spacetime is therefore a flat Minkowski one, and in Cartesian coordinates it is simply
\begin{equation}
	ds^2 = g_{\mu\nu}dx^\mu dx^\nu = -dt^2 + dx^2 + dy^2 + dz^2\,.
\end{equation}
It is convenient to introduce the so-called laboratory observer, defined by the four velocity
\begin{equation}
	n^\mu=(1,0,0,0), \hspace{2cm}n_\mu=(-1,0,0,0)\,.
\end{equation}
We now briefly recall the usual relativistic tensor formalism of the Maxwell equations in special relativity, before
showing the necessary changes to obtain its GLM version. Even if we are assuming a flat spacetime in Cartesian coordinates,
with Christoffel symbols that are identically zero, we nevertheless write the equations by adopting covariant derivatives,
in view of possible future extensions to intrinsically curved spacetimes.
At the end of this Section we reach the conclusion that the system \eqref{eqn.B}-\eqref{eqn.psi}
is already  consistent with special relativity. However, this must be 
shown through a proper discussion.

\subsection{Pure Maxwell} 

We write the Maxwell tensor $F^{\mu\nu}$ and its dual, the Faraday tensor $\FD^{\mu\nu}=\frac{1}{2}\epsilon^{\mu\nu\alpha\beta}F_{\alpha\beta}$, as
\begin{eqnarray}
F^{\mu\nu} & = & n^{\,\mu}E^{\nu} - E^{\mu}n^{\nu} + \epsilon^{\,\mu\nu\lambda\kappa}B_{\lambda}n_{\kappa}, 
\label{eq:F} \\
\FD^{\mu\nu} & = & n^{\,\mu}B^{\nu} - B^{\mu}n^{\nu}  - \epsilon^{\,\mu\nu\lambda\kappa}E_{\lambda}n_{\kappa}\,,
\label{eq:F*}
\end{eqnarray}
where $\epsilon^{\,\mu\nu\lambda\kappa}$ is the four dimensional Levi-Civita tensor.
We recall that  $E^\mu$ and  $B^\mu$, are both spatial four-vectors, with zero temporal components, and they express
the electric field and the magnetic field as measured by the laboratory observer.
The Maxwell tensor can also be written in terms of the four-vector potential $A_\mu$ as
\begin{equation}
	\label{F-vectorA}
	F_{\mu\nu} = 2 \nabla_{[\mu} A_{\nu]} = \nabla_{\mu} A_{\nu} - \nabla_{\nu} A_{\mu} = \partial_{\mu} A_{\nu} - \partial_{\nu} A_{\mu}\,.
\end{equation}
The first couple of Maxwell equations is a purely formal consequence of the definitions introduced so far. In fact, if we contract \eqref{F-vectorA} with the Levi-Civita, we get
\begin{equation}
	\label{Pure-Maxwell-first-couple}
	\epsilon^{\mu\nu\alpha\beta}F_{\mu\nu}=2\epsilon^{\mu\nu\alpha\beta}\nabla_{[\mu} A_{\nu]} \hspace{0.3cm} 
\Longrightarrow	\hspace{0.3cm} \FD^{\alpha\beta}=\epsilon^{\mu\nu\alpha\beta}\nabla_{\mu} A_{\nu}
\hspace{0.3cm}\Longrightarrow	\hspace{0.3cm}\nabla_\beta \FD^{\alpha\beta}=0\,.
\end{equation}
It is easy to show (see  \cite{Landau:1975pou}) that the four equations expressed by \eqref{Pure-Maxwell-first-couple}
correspond to the standard divergence-free condition for the magnetic field
plus the evolution of the magnetic field.
The Lagrangian density must be a relativistic invariant, and it is given by 
\begin{equation}
	\Lambda=- \quart F_{\mu\nu} F^{\mu\nu} = \halb (E^2-B^2)\,.
\end{equation}
The second couple of Maxwell equations is obtained as an Euler Lagrange equation with respect to the vector potential $A_\mu$. Namely:
\begin{equation}
	\partial_\mu\left(\frac{\partial\Lambda}{\partial (\partial_\mu A_\nu)}\right)-\frac{\partial\Lambda}{\partial A_\nu}=0\,.
\end{equation}
After a few calculations this provides
\begin{eqnarray}
	\label{MaxGLM-rel-2}
	\nabla_\nu F^{\mu\nu}  =0\,, 
\end{eqnarray}
which corresponds to the standard divergence-free condition for the electric field
in vacuum, plus the evolution of the electric field.
%

\subsection{The relativistic Maxwell-GLM}
In the new framework, the Maxwell and the Faraday tensors are modified as
\begin{eqnarray}
	F_{\mu\nu}  &=& \nabla_{\mu} A_{\nu} -  \nabla_{\nu} A_{\mu} = n_{\,\mu}E_{\nu} - E_{\mu}n_{\nu} + \epsilon_{\,\mu\nu\lambda\kappa}B^{\lambda}n^{\kappa}+\epsilon_{\,\mu\nu\alpha\beta}n^\alpha
	g^{\beta \gamma} \nabla_\gamma Z, \\
	\FD_{\mu\nu}  &=& \frac{1}{2}\epsilon_{\,\mu\nu\alpha\beta}F^{\alpha\beta} = n_\mu B_\nu - n_\nu B_\mu + {\epsilon_{\,\mu\nu\alpha\beta}}n^\alpha E^\beta - \left( n_{\mu}\nabla_{\nu}Z - n_{\nu}\nabla_{\mu}Z \right) \,,
\end{eqnarray}
from which it follows that the purely spatial magnetic and electric fields can be obtained as
\begin{eqnarray}
	\label{newB}
	B_{i}  &=& \FD_{i\nu}n^\nu + \nabla_i Z = (\nabla \times \A)_i + \nabla_i Z, \\
	\label{newE}
	E_{i}  &=& F_{i\nu}n^\nu\,. 
\end{eqnarray}
We stress that, according to \eqref{newB}-\eqref{newE}, only the magnetic field
is affected by the extra terms in the electromagnetic tensors. The modified
Lagrangian density is 
\begin{eqnarray}
	\Lambda&=&- \quart F_{\mu\nu} F^{\mu\nu} - \frac{1}{2} (\nabla_\mu A^\mu)^2 - \frac{1}{2}g^{\mu\nu}\nabla_\mu Z\, \nabla_\nu Z\nonumber\\
	&=& \halb (E^2-(\nabla \times \A)^2) -  \frac{1}{2}(\nabla_\mu A^\mu)^2 - \frac{1}{2}g^{\mu\nu}\nabla_\mu Z\, \nabla_\nu Z\,.
\end{eqnarray}
\subsubsection{First couple of Maxwell-GLM}

The first couple of Maxwell equations is obtained as the sum of two formal identities plus one
Euler Lagrange equation. The first formal identity is 
\begin{equation}
	\label{formal-1}
	\nabla_\nu \FD^{\mu\nu} =0\,.
\end{equation}
The second formal identity is
\begin{eqnarray}
\label{formal-2}
\nabla_\nu ( n^\nu\nabla^\mu Z + g^{\mu\nu}\phi) =0 \,,
\end{eqnarray}
which follows directly from the definition $\phi=-\partial_t Z$ and the fact that
$\nabla_\mu n^\mu=0$.
Finally, we need the  Euler Lagrange equation with respect to the field $Z$. Namely:
\begin{equation}
	\partial_\mu\left(\frac{\partial\Lambda}{\partial (\partial_\mu Z)}\right)-\frac{\partial\Lambda}{\partial Z}=0\,.
\end{equation}
from which we obtain a wave equation for $Z$:
\begin{equation}
	\label{eq.wave.Z}
	g^{\alpha\beta}\nabla_\alpha \nabla_\beta Z=0\,.
\end{equation}	
Now, by subtracting \eqref{formal-2} from \eqref{formal-2}, we obtain	
\begin{eqnarray}
	\label{MaxGLM-rel-1}
	\nabla_\nu (\FD^{\mu\nu} - n^\nu\nabla^\mu Z   - g^{\mu\nu}\phi) =0 \,.
\end{eqnarray}
It is easy to show that 	\eqref{MaxGLM-rel-1}
is the  covariant relativistic version of 
\eqref{eqn.B}-\eqref{eqn.phi}, where the wave equation \eqref{eq.wave.Z} must be used.
In fact, by first considering the temporal component of \eqref{MaxGLM-rel-1}, we can find
\begin{eqnarray}
&&	\partial_i \FD^{t i} - \partial_\nu(n^\nu\partial^t Z) - g^{t\nu}\partial_\nu\phi=0\nonumber\\
\Longrightarrow&&\partial_i(n^t B^i - n^t\partial^i Z) +\partial_t^2 Z + \partial_t\phi=0\nonumber\\
\Longrightarrow&&\partial_i B^i - \partial_i\partial^i Z  +\partial_t^2 Z + \partial_t\phi=0  \hspace{1cm}\textrm{use \eqref{eq.wave.Z}}\nonumber   \\
\Longrightarrow&&\partial_i B^i  + \partial_t\phi=0\,,
\end{eqnarray}
which is \eqref{eqn.phi} with $c_h=1$.
On the other hand, by considering the spatial component of \eqref{MaxGLM-rel-1}, we obtain
\begin{eqnarray}
	&&	\partial_\nu \FD^{i\nu} + \partial_\nu(n^\nu\partial^i Z) - g^{ij}\partial_j\phi=0\nonumber\\
	\Longrightarrow&&\partial_t(-n^t B^i + n^t\partial^i Z) +
\partial_j(\epsilon^{ij\alpha\beta} n_\alpha E_\beta ) -	\partial_t(n^t  \partial^i Z) - g^{ij}\partial_j\phi=0\nonumber\\
	\Longrightarrow&&-\partial_t B^i + \partial_t\partial^i Z -\epsilon^{ijk}\partial_j E_k  -\partial_t\partial^i Z - g^{ij}\partial_j\phi=0  \nonumber   \\
	\Longrightarrow&&\partial_t B^i +\epsilon^{ijk}\partial_j E_k + g^{ij}\partial_j\phi=0\,,
\end{eqnarray}
which is \eqref{eqn.B} with $c_0=c_h=1$.

\subsubsection{Second couple of Maxwell-GLM}
The second couple of Maxwell equations is obtained as an Euler Lagrange equation with respect to the fields $A_\mu$. Namely:
\begin{equation}
	\partial_\mu\left(\frac{\partial\Lambda}{\partial (\partial_\mu A_\nu)}\right)-\frac{\partial\Lambda}{\partial A_\nu}=0\,.
\end{equation}
After a few calculations this provides
\begin{eqnarray}
	\label{MaxGLM-rel-2}
	\nabla_\nu (F^{\mu\nu} - g^{\mu\nu}\psi) =0\,, 
\end{eqnarray}
which is the 
covariant relativistic version of 
\eqref{eqn.E}-\eqref{eqn.psi}, where $\psi= \nabla_\mu A^\mu$. In fact, the temporal component of 
\eqref{MaxGLM-rel-2} provides
\begin{eqnarray}
	&&	\partial_\nu F^{t\nu}  - g^{t\nu}\partial_\nu\psi=0\nonumber\\
	\Longrightarrow&&\partial_i F^{ti}  - g^{tt}\partial_t\psi=0\nonumber\\
	\Longrightarrow&&\partial_i E^i + \partial_t \psi=0\,,
\end{eqnarray}
which is \eqref{eqn.psi} with $c_h=1$. Finally, the spatial component of 
\eqref{MaxGLM-rel-2} gives 
\begin{eqnarray}
	&&	\partial_\nu F^{i\nu}  - g^{i\nu}\partial_\nu\psi=0\nonumber\\
	\Longrightarrow&&\partial_t(-n^t E^i ) +
\partial_j(\epsilon^{ij\alpha\beta} B_\alpha n_\beta +\epsilon^{ij\alpha\beta} n_\alpha \partial_\beta Z) -g^{ij}\partial_j\psi=0\nonumber\\
	\Longrightarrow&&-\partial_t E^i  +
\epsilon^{ijk}\partial_j  B_k -\epsilon^{ijk} \partial_j\partial_k  Z
 -g^{ij}\partial_j\psi=0\nonumber\\
	\Longrightarrow&&\partial_t E^i  -
\epsilon^{ijk}\partial_j  B_k +g^{ij}\partial_j\psi=0\,,
\end{eqnarray}
which is \eqref{eqn.E} with $c_0=c_h=1$.

\section{Compatible numerical discretizations}
\label{sec.numerics}

\subsection{Exactly energy-conserving semi-discrete numerical schemes on collocated meshes}
\label{sec.num.htc}

Throughout this section we will employ the Einstein summation convention and use the following compact notation for the hyperbolic PDE system \eqref{eqn.B}-\eqref{eqn.psi}: 
\begin{equation}
	\partial_t \q + \partial_k \mathbf{f}_k(\q) = 0, 	 
	\label{eqn.pde} 
\end{equation}
where $\q = \left(  \B, \phi,\E, \psi \right)$ is the state vector and $\f_k$ is the flux tensor. Thanks to the compatibility of \eqref{eqn.pde} with the extra conservation law for the total energy \eqref{eqn.energy}, which 
can be written more compactly with the total energy flux $F_k$ according to \eqref{eqn.energy} as 
\begin{equation}
	\partial_t \mathcal{E} + \partial_k F_k = 0, 
	\label{eqn.pde.extra} 
\end{equation}
the following identity holds at the continuous level:
\begin{equation}
	\p \cdot \partial_k \mathbf{f}_k(\q) = \partial_k F_k.   
	\label{eqn.flux.comp} 
\end{equation}
Based on the ideas outlined in \cite{HTCGPR,HTCMHD,HTCAbgrall,HTCAbgrall2,HTCTwoFluid} in the following we show how to achieve this compatibility property exactly also at the discrete level. 
The computational domain $\Omega \subset \mathbb{R}^d$ in $d$ space dimensions could in principle be paved by a fairly general mesh of \textit{orthogonal} polygonal / polyhedral control volumes denoted by $\Omega^{\ell}$ in the following. However, in this paper we limit ourselves to the simple Cartesian case. The common edge / face of two neighboring polygons / polyhedra $\Omega^{\ell}$ and $\Omega^{\err}$ is $\partial \Omega^{\ell \err} = \Omega^{\ell} \cap \Omega^{\err}$ and $\mathcal{N}_\ell$ is the set of neighbors of the element $\Omega^{\ell}$. 
The compatible semi-discrete finite volume scheme for the discretization of the hyperbolic system of conservation laws \eqref{eqn.pde} with extra conservation law \eqref{eqn.pde.extra}, \eqref{eqn.flux.comp} reads for each control volume $\Omega^{\ell}$ as follows: 
\begin{equation}
	\frac{\partial \q^{\ell}}{\partial t}  = -  \sum_{\Omega^\err \in \mathcal{N}_{\ell}}
	\frac{\left|\partial\Omega^{\ell\err}\right|}{\left|\Omega^{\ell}\right|}  \f^{\ell\err}. 
	\label{eqn.flux2dfv}
\end{equation}	
Throughout this paper the scheme \eqref{eqn.flux2dfv} is integrated in time via a high order Runge-Kutta method of suitable order greater or equal than four.
The compatible numerical flux in normal direction is denoted by $\f^{\ell\err}$ and must satisfy the following discrete compatibility condition, which is a discrete analogy to the continuous identity \eqref{eqn.flux.comp}:    
\begin{equation}
	\p^{\ell} \cdot \left( \f^{\ell\err} - \f_k^{\ell} n_k^{\ell\err} \right) +
	\p^{\err} \cdot \left( \f_k^{\err} n_k^{\ell\err} - \f^{\ell\err} \right) =  \left(F_k^{\err}  - F_k^{\ell} \right) n_k^{\ell\err}  ,\label{eqn.compatibility_2dFV}
\end{equation} 
with $\mathbf{n}^{\ell \err} = \left\{ n_k^{\ell \err} \right\}$ the unit normal vector pointing from element $\Omega^{\ell}$ to its neighbor $\Omega^{\err}$. 
Therefore, the following identity is obviously true: $\mathbf{n}^{\err\ell} = -\mathbf{n}^{\ell \err}$.
If a numerical flux $\f^{\ell\err}$ satisfies \eqref{eqn.compatibility_2dFV} then it is easy to prove total energy conservation. We thus have the following 
\begin{theorem}[Total energy conservation]
	\label{thm.energy.htc} 
	The finite volume scheme \eqref{eqn.flux2dfv} with a numerical flux $\f^{\ell\err}$ that satisfies \eqref{eqn.compatibility_2dFV} conserves total energy in the sense that, for vanishing boundary fluxes, we have 
	\begin{equation}
		\int \limits_{\Omega} \frac{\partial \mathcal{E}}{\partial t} \mathrm{d}\x= 0.
	\end{equation}
\end{theorem}
\begin{proof}
	We take the dot product of the semi-discrete scheme  \eqref{eqn.flux2dfv} with the discrete main field variables $\p^{\ell}$, multiply by the cell volume $\left|\Omega^{\ell}\right|$ and obtain: 
	\begin{eqnarray*}
		\left|\Omega^{\ell}\right| \p^{\ell}\cdot	\frac{\partial \q^{\ell}}{\partial t} = - \sum_{\Omega^\err \in \mathcal{N}_{\ell}}
		\left|\partial\Omega^{\ell\err}\right|  \p^{\ell} \cdot \f^{\ell\err}.
	\end{eqnarray*}
	Adding and subtracting $\displaystyle\sum_{\Omega^\err \in \mathcal{N}_{\ell}}\left|\partial\Omega^{\ell\err}\right|\halb \p^{\err} \cdot \f^{\ell\err}$ leads to 
	\begin{equation*}
		\left|\Omega^{\ell}\right| \frac{\partial \mathcal{E}^{\ell}}{\partial t} = 
		-  \sum_{\Omega^\err \in \mathcal{N}_{\ell}}
		\left|\partial\Omega^{\ell\err}\right| \halb \left( \left(  \p^{\ell}+\p^{\err}\right) \cdot \f^{\ell\err}   	
		 + \left(  \p^{\ell}-\p^{\err}\right) \cdot \f^{\ell\err}   \right).
	\end{equation*}
	Using the compatibility condition \eqref{eqn.compatibility_2dFV} leads to  
	\begin{equation*}
		\left|\Omega^{\ell}\right| \frac{\partial \mathcal{E}^{\ell}}{\partial t} = 
		-  \sum_{\Omega^\err \in \mathcal{N}_{\ell}}
		\left|\partial\Omega^{\ell\err}\right| \halb \left(  \left(F_k^{\err}  - F_k^{\ell} \right) n_k^{\ell \err} -		
		\left(\p^{\err}\cdot \f^{\err}_{k} - \p^{\ell}\cdot \f^{\ell}_{k} \right) n_k^{\ell \err}
		+  \left(  \p^{\ell}+\p^{\err}\right) \cdot 
		\f^{\ell\err} \right).
	\end{equation*}
	Since the integral of the normal vector over a closed surface vanishes, the following identity holds: 
	\begin{equation} 
		\sum_{\Omega^\err \in \mathcal{N}_{\ell}} \left|\partial\Omega^{\ell\err}\right| \mathbf{n}^{\ell \err} = 0.
		\label{eqn.closedsurface2d} 
	\end{equation}
	Adding $\p^{\ell}\cdot\f^{\ell}_k + F_k^\ell$ multiplied by \eqref{eqn.closedsurface2d} one obtains 
	\begin{eqnarray*}
		\left|\Omega^{\ell}\right| \frac{\partial \mathcal{E}^{\ell}}{\partial t} &=& 
		- \sum_{\Omega^\err \in \mathcal{N}_{\ell}}
		\left|\partial\Omega^{\ell\err}\right| \halb \left( \left(F_k^{\err}  + F_k^{\ell} \right) n_k^{\ell \err}		
		- \left(\p^{\err}\cdot \f^{\err}_{k}+ \p^{\ell}\cdot \f^{\ell}_{k}\right)n_k^{\ell \err} 
		+ \left(  \p^{\ell}+\p^{\err}\right) \cdot\f^{\ell\err} \right) \nonumber \\ 
		&=& 
		-\sum_{\Omega^\err \in \mathcal{N}_{\ell}} \left|\partial\Omega^{\ell\err}\right| F^{\ell\err},
	\end{eqnarray*}
	with the numerical total energy flux in the normal direction $\normal^{\ell\err}$ given by 
	\begin{equation} 
		F^{\ell\err} = \halb \left( \left(F_k^{\err}  + F_k^{\ell} \right) n_k^{\ell \err}		
		- \left(\p^{\err}\cdot \f^{\err}_{k}+ \p^{\ell}\cdot \f^{\ell}_{k}\right)n_k^{\ell \err} 
		+ \left(  \p^{\err} + \p^{\ell} \right) \cdot\f^{\ell\err} \right).
	\end{equation} 	
	The total energy conservation is then finally obtained by summing up over all elements $\Omega^{\ell}$, assuming that the fluxes are zero at the boundary of the domain and by employing the telescopic sum property, since the sum of all numerical total energy fluxes $F^{\ell\err}$ at the internal interfaces cancels 
	\begin{equation*}
		\int \limits_{\Omega} \frac{\partial \mathcal{E}}{\partial t} \mathrm{d}\x= \sum_{\ell}  \left|\Omega^{\ell}\right| \frac{\partial \mathcal{E}^{\ell}}{\partial t}= 0. 
	\end{equation*}
\end{proof}

\paragraph{Compatible Abgrall-type scheme}

As in \cite{Abgrall2018,HTCAbgrall,HTCAbgrall2,HTCTwoFluid} the \textit{thermodynamically compatible} \textbf{Abgrall flux} reads  
\begin{equation}
	\f^{\ell\err} = 
	 \halb \left( \mathbf{f}_k^{\ell} + \mathbf{f}_k^{\err} \right) n_k^{\ell \err}  
	 - \alpha^{\ell \err}  	 \left( \p^{\err} - \p^{\ell} \right), 	
	\label{eqn.abgrallflux} 
\end{equation}
with the scalar correction factor $\alpha^{\ell \err}$ that is directly obtained by imposing the discrete compatibility condition \eqref{eqn.compatibility_2dFV} on the numerical flux and which is given by  
\begin{equation}
	\alpha^{\ell \err} =  \frac{ \left( F^{\err}_k - F^{\ell}_k \right) n_k^{\ell \err} + 
	    \halb \left(  \p^{\err} + \p^{\ell}  \right) \cdot   \left( \mathbf{f}_k^{\ell} - \mathbf{f}_k^{\err} \right) n_k^{\ell \err}  }{ \left(  \p^{\err} - \p^{\ell}  \right)^2 }.  
	\label{eqn.alpha2d}  		
\end{equation}	
One can easily verify that the Abgrall flux \eqref{eqn.abgrallflux} with the correction factor \eqref{eqn.alpha2d} satisfies the discrete compatibility condition \eqref{eqn.compatibility_2dFV} \textit{by construction}, since the correction factor $\alpha^{\ell \err}$ is obtained by imposing \eqref{eqn.compatibility_2dFV} on \eqref{eqn.abgrallflux}.

\subsection{Fully-discrete structure-preserving staggered semi-implicit scheme}
\label{sec.num.si}
In this section we present a completely different compatible discretization compared to the previous one. The method introduced in this section is fully-discrete, globally exactly energy-conserving and uses a \textit{staggered mesh} with \textit{mimetic} discrete differential operators that preserve the essential vector calculus identities exactly at the discrete level. The computational domain is denoted by $\Omega$. Furthermore, we denote the control volumes of the primary mesh by $\Omega_c$ with cell centers $\x_c$ and vertices $\x_p$, while the control volumes of the dual mesh are denoted by $\Omega_p$, with cell centers $\x_p$ (see Fig.\ref{fig:Grid}.)    

\begin{figure}[!h]
	\centering
\includegraphics[width=0.45\textwidth]{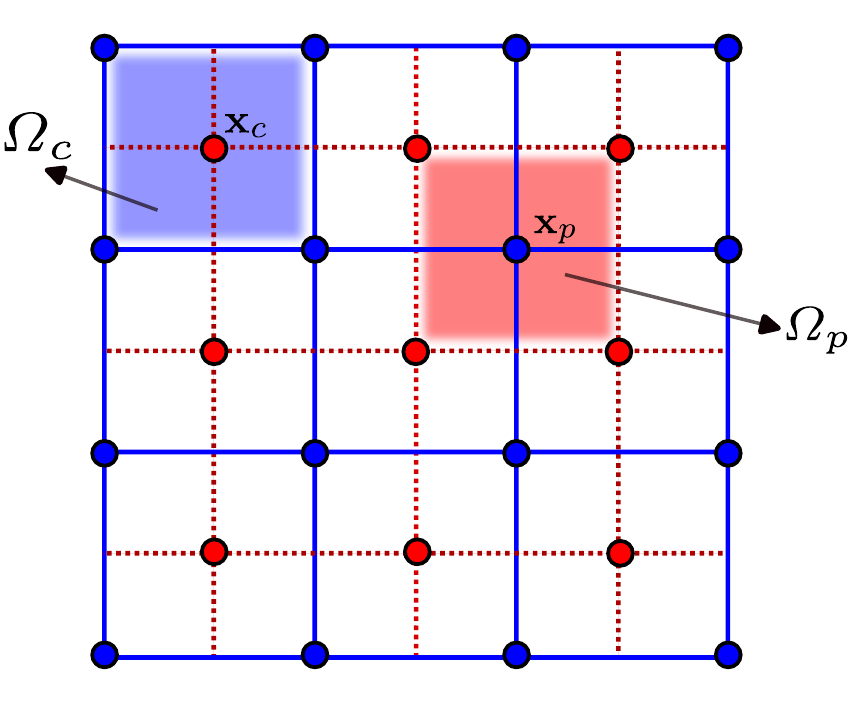}
	\caption{Schematic representation of the primary $\Omega_c$ and of the dual $\Omega_p$ mesh.}
	\label{fig:Grid}
\end{figure}

In the following we introduce two mimetic discrete nabla operators, see \cite{HymanShashkov1997,Maire2007,Maire2008,Maire2009}: 
\begin{equation}
	\nabla_c^p = \frac{1}{|\Omega_c|} \sum_{p \in \Omega_c} l_{pc} \mathbf{n}_{pc} \,,
	\label{eqn.nablacp}  
\end{equation} 	
and its dual
\begin{equation}
	\nabla_p^c = \frac{1}{|\Omega_p|} \sum_{c \in \Omega_p} l_{pc} \mathbf{n}_{cp},
	\label{eqn.nablapc}  
\end{equation} 	
with $\mathbf{n}_{cp} = -\mathbf{n}_{pc}$. Note that in \eqref{eqn.nablacp} and \eqref{eqn.nablapc} the discrete nabla operators have not been applied to any field yet. Their concrete application to scalar and discrete vector fields will be introduced below.  
Due to the Gauss theorem the following identities are obviously true: 
\begin{equation}
	\sum_{p \in \Omega_c} l_{pc} \mathbf{n}_{pc} = 0, \qquad 
	\sum_{c \in \Omega_p} l_{pc} \mathbf{n}_{cp} = 0. 
	\label{eqn.gauss.disc} 
\end{equation} 

For scalar fields $\phi_p=\phi(\x_p)$ and vector fields $\A_p=\A(\x_p)$ defined on the vertices of the primal grid and for scalar and vector fields defined on the centers of the primal grid (the vertices of the dual grid)  $\phi_c=\phi(\x_c)$ and $\A_c=\A(\x_c)$ the following discrete gradient, divergence and curl operators and their duals are naturally defined via \eqref{eqn.nablacp} and \eqref{eqn.nablapc}: 
\begin{alignat}{2}
	&\nabla_c^p \phi_p   = \frac{1}{|\Omega_c|} \sum_{p \in \Omega_c} l_{pc} \mathbf{n}_{pc} \,     \phi_p,  \qquad 
	&&\nabla_p^c \phi_c      = \frac{1}{|\Omega_p|} \sum_{c \in \Omega_p} l_{pc} \mathbf{n}_{cp} \,     \phi_c, 
	\\  	
	&\nabla_c^p \cdot  \A_p = \frac{1}{|\Omega_c|} \sum_{p \in \Omega_c} l_{pc} \mathbf{n}_{pc} \cdot  \A_p,  \qquad
	&&\nabla_p^c \cdot  \A_c = \frac{1}{|\Omega_p|} \sum_{c \in \Omega_p} l_{pc} \mathbf{n}_{cp} \cdot  \A_c, 
	\\ 
	&\nabla_c^p \times \A_p = \frac{1}{|\Omega_c|} \sum_{p \in \Omega_c} l_{pc} \mathbf{n}_{pc} \times \A_p,  \qquad
	&&\nabla_p^c \times \A_c = \frac{1}{|\Omega_p|} \sum_{c \in \Omega_p} l_{pc} \mathbf{n}_{cp} \times \A_c. 
\end{alignat} 
Furthermore, after some calculations one obtains the following discrete vector calculus identities, which are the fundamental pillar of the proposed numerical scheme: 
The discrete curl applied to the discrete gradient of a discrete scalar field $\phi$ vanishes identically, i.e. 
\begin{equation}
	\nabla_c^p \times \nabla_p^c \, \phi_c = 0, \qquad 
	\nabla_p^c \times \nabla_c^p \, \phi_p = 0,  
	\label{eqn.rot.grad} 
\end{equation}
which is the discrete analogue of the continuous identity $\nabla \times \nabla \phi = 0$. Furthermore, also the discrete divergence applied to the discrete curl of a vector field $\A$ vanishes: 
\begin{equation}
	\nabla_c^p \cdot \nabla_p^c \times \A_c = 0, \qquad 
	\nabla_p^c \cdot \nabla_c^p \times \A_p = 0,  
	\label{eqn.div.rot} 
\end{equation}
which reflects the continuous identity $\nabla \cdot \nabla \times \mathbf{A} = 0 $ at the discrete level. The extended calculations are shown in the Appendix.

The discrete magnetic field and the discrete cleaning scalar $\psi$ are governed
by the differential identities \eqref{eqn.B2}, \eqref{eqn.psi2} and are located
at the cell centers of the primal mesh and are denoted by $\B_c^n=\B(\x_c,t^n)$
and $\psi^n_c=\psi(\x_c,t^n)$, respectively, while the discrete electric field and
the discrete cleaning scalar $\phi$ are governed by the Euler-Lagrange equations \eqref{eqn.phi2}, \eqref{eqn.E2} and are located in the vertices of the primal
mesh, i.e. the cell centers of the dual mesh, and are denoted by
$\E_p^n=\E(\x_p,t^n)$ and $\phi^n_p=\phi(\x_p,t^n)$, respectively. Furthermore, we
use the notations 
$$\B_c^{n+\halb} = \halb (\B_c^{n}+\B_c^{n+1}), \qquad \E_p^{n+\halb} = \halb (\E_p^{n}+\E_p^{n+1}),$$ 
$$\psi_c^{n+\halb} = \halb (\psi_c^{n}+\psi_c^{n+1}), \qquad \phi_p^{n+\halb} = \halb (\phi_p^{n}+\phi_p^{n+1}).$$
We now propose the following compatible discretization of \eqref{eqn.B}-\eqref{eqn.psi}: 
\begin{eqnarray}
	\B_c^{n+1} &=& \B_c^{n} - \Delta t \, c_0 \nabla_c^p \times \E_p^{n+\halb} - \Delta t \, c_h \nabla_c^p \phi_p^{n+\halb}, 
	\label{eqn.B.fd} 
	\\ 
	\phi_p^{n+1} &=& \phi_p^{n} - \Delta t \, c_h \nabla_p^c \cdot \B_c^{n+\halb},  
	\label{eqn.phi.fd} 
	\\
	\E_p^{n+1} &=& \E_p^{n} + \Delta t \, c_0 \nabla_p^c \times \B_c^{n+\halb} - \Delta t \, c_h \nabla_p^c \psi_c^{n+\halb},
	\label{eqn.E.fd} 
	\\
	\psi_c^{n+1} &=& \psi_c^{n} - \Delta t \, c_h \nabla_c^p \cdot \E_p^{n+\halb}.
	\label{eqn.psi.fd} 
\end{eqnarray}
Inserting \eqref{eqn.B.fd} into \eqref{eqn.phi.fd} and using the discrete vector identity \eqref{eqn.div.rot} one obtains the following discrete wave equation for the cleaning scalar $\phi$: 
\begin{equation}
	\phi_p^{n+1} - \quart \Delta t^2 c_h^2 \nabla_p^c \cdot \nabla_c^p \phi_p^{n+1} = \phi_p^{n} - \Delta t \, c_h \nabla_p^c \cdot \B_c^{n} + \quart \Delta t^2 c_h^2 \nabla_p^c \cdot \nabla_c^p \phi_p^{n}.
	\label{eqn.phi.wave}
\end{equation} 
Inserting \eqref{eqn.B.fd} and \eqref{eqn.psi.fd} into \eqref{eqn.E.fd} we obtain
a discrete vector wave equation for the electric field 
\begin{equation}
	 \E_p^{n+1} + \quart \Delta t^2 \, c_0^2 \nabla_p^c \times \nabla_c^p \times \E_p^{n+1} 
	 - \quart \Delta t^2 \, c_h^2 \nabla_p^c \nabla_c^p \cdot \E_p^{n+1} = \mathbf{r}_p,
	 \label{eqn.E.wave} 
\end{equation} 
with the known right hand side 
\begin{equation}
\mathbf{r}_p = \E_p^{n} + \Delta t \, c_0 \nabla_p^c \times \B_c^{n} - \quart \Delta t^2 \, c_0^2 \nabla_p^c \times \nabla_c^p \times \E_p^{n} - \Delta t \, c_h \nabla_p^c \psi_c^{n}
+ \quart \Delta t^2 \, c_h^2 \nabla_p^c \nabla_c^p \cdot \E_p^{n}. 
\label{eqn.E.wave.rhs} 
\end{equation}
The above substitutions correspond to the use of the Schur complement in the linear system \eqref{eqn.B.fd}-\eqref{eqn.psi.fd}. 
Once the new cleaning scalar $\phi$ and the new electric field have been obtained, the magnetic field and the cleaning scalar $\psi$ can be directly updated via \eqref{eqn.B.fd} and \eqref{eqn.psi.fd}. This procedure is analogous to the post-projection stage in semi-implicit schemes for the incompressible Navier-Stokes equations \cite{HW65,patankar,patankar2} or in semi-implicit schemes for shallow water flows \cite{Casulli1990,CasulliCheng1992}. It is very interesting to note that the final equations to be solved are discrete second order wave equations for the electric field and for the cleaning variable $\phi$, respectively, and thus reflect the solution of the original Euler-Lagrange equations \eqref{eqn.el.A} and \eqref{eqn.el.phi} on a discrete level; not for the original potentials $\A$ and $Z$, but for their time derivatives $\E$ and $\phi$. 

In the following we study the total energy conservation and the asymptotic-preserving (AP) property of the proposed structure-preserving semi-implicit scheme. 

\begin{theorem}
	\label{thm.energy.simm} 	
In the case of periodic boundaries the scheme \eqref{eqn.B.fd}-\eqref{eqn.psi.fd} conserves the global discrete total energy 
$$\mathcal{E}^n = 
\sum \limits_{c \in \Omega} |\Omega_c| \halb (\B_c^n)^2 +
\sum \limits_{p \in \Omega} |\Omega_p| \halb (\phi_p^n)^2 +
\sum \limits_{p \in \Omega} |\Omega_p| \halb (\E_p^n)^2 +
\sum \limits_{c \in \Omega} |\Omega_c| \halb (\psi_c^n)^2
$$ 
exactly, i.e. 
$$\mathcal{E}^{n+1} = \mathcal{E}^n.$$  
\begin{proof}
	Computing the dot product of \eqref{eqn.B.fd} with $\B_c^{n+\halb}$ and of \eqref{eqn.E.fd} with $\E_p^{n+\halb}$, multiplying \eqref{eqn.phi.fd}  with $\phi_p^{n+\halb}$ and \eqref{eqn.psi.fd} with  $\psi_c^{n+\halb}$ yields 
\begin{eqnarray}
	|\Omega_c| \, \B_c^{n+\halb} \cdot \left( \B_c^{n+1} - \B_c^{n} \right) \!\! &=& \!\! - \Delta t \, \B_c^{n+\halb} \cdot  \left( c_0 \sum_{p \in \Omega_c} l_{pc} \mathbf{n}_{pc} \times \E_p^{n+\halb} + c_h \sum_{p \in \Omega_c} l_{pc} \mathbf{n}_{pc} \, \phi_p^{n+\halb} \right)\!\!, \nonumber 
	\label{eqn.B.dp} 
	\\ 
	|\Omega_p| \, \phi_p^{n+\halb} \left( \phi_p^{n+1} -  \phi_p^{n} \right) \!\! &=& \!\! - \Delta t \, \phi_p^{n+\halb} \, c_h \sum_{c \in \Omega_p} l_{pc} \mathbf{n}_{cp} \cdot \B_c^{n+\halb},  \nonumber  
	\label{eqn.phi.dp} 
	\\
	|\Omega_p| \, \E_p^{n+\halb} \cdot \left( \E_p^{n+1} - \E_p^{n} \right) \!\! &=& \!\! + \Delta t \, \E_p^{n+\halb} \cdot \left( c_0 \sum_{c \in \Omega_p} l_{pc} \mathbf{n}_{cp} \times \B_c^{n+\halb} - c_h \sum_{c \in \Omega_p} l_{pc} \mathbf{n}_{cp} \, \psi_c^{n+\halb} \right)\!\!, \nonumber 
	\label{eqn.E.dp} 
	\\
	|\Omega_c| \, \psi_c^{n+\halb} \left( \psi_c^{n+1} - \psi_c^{n} \right) \!\! &=& \!\! - \Delta t \, \psi_c^{n+\halb} \, c_h  \sum_{p \in \Omega_c} l_{pc} \mathbf{n}_{pc} \cdot \E_p^{n+\halb}. \nonumber 
	\label{eqn.psi.dp} 
\end{eqnarray}	
Since the quantities $\B_c^{n+\halb}$ and $\psi_c^{n+\halb}$ do not depend on the index $p$ and since $\phi_p^{n+\halb}$ and $\E_p^{n+\halb}$ do not depend on $c$ we can take these terms on the right hand side into the sums. Furthermore, we have that 
$$\B_c^{n+\halb} \cdot \left( \B_c^{n+1} - \B_c^{n} \right) = \halb (\B_c^{n+1})^2 - \halb (\B_c^{n})^2,$$ 
$$\phi_p^{n+\halb} \left( \phi_p^{n+1} -  \phi_p^{n} \right) = \halb (\phi_p^{n+1})^2 - \halb (\phi_p^{n})^2,$$ 
$$\E_p^{n+\halb} \cdot \left( \E_p^{n+1} - \E_p^{n} \right) = \halb (\E_p^{n+1})^2 - \halb (\E_p^{n})^2,$$
$$\psi_c^{n+\halb} \left( \psi_c^{n+1} -  \psi_c^{n} \right) = \halb (\psi_c^{n+1})^2 - \halb (\psi_c^{n})^2.$$ 
We thus obtain
\begin{eqnarray}
	|\Omega_c| \! \left( \halb (\B_c^{n+1})^2 - \halb (\B_c^{n})^2 \right) \!\!\!\!\!\! &=& \!\!\!\!\!  - \Delta t \!  \left( \! c_0 \!\! \sum_{p \in \Omega_c} \!\! l_{pc} \mathbf{n}_{pc} \cdot \E_p^{n+\halb} \! \times \B_c^{n+\halb} \! + c_h \!\! \sum_{p \in \Omega_c} \!\! l_{pc} \mathbf{n}_{pc} \! \cdot  \phi_p^{n+\halb} \B_c^{n+\halb} \! \right)\!\!, \nonumber 
	\label{eqn.B.dp} 
	\\ 
	|\Omega_p| \left( \halb (\phi_p^{n+1})^2 - \halb (\phi_p^{n})^2 \right) \!\!\!\!\! &=&\!\!\!\!\! - \Delta t \, c_h \sum_{c \in \Omega_p} l_{pc} \mathbf{n}_{cp} \cdot \B_c^{n+\halb} \phi_p^{n+\halb},  \nonumber  
	\label{eqn.phi.dp} 
	\\
	|\Omega_p|  \!  \left( \halb (\E_p^{n+1})^2 - \halb (\E_p^{n})^2 \right) \!\!\!\!\!\! &=& \!\!\!\!\!  + \Delta t \! \left(\! c_0 \!\! \sum_{c \in \Omega_p}\! l_{pc} \mathbf{n}_{cp} \cdot \B_c^{n+\halb} \! \times \E_p^{n+\halb} \! - c_h \!\! \sum_{c \in \Omega_p} \! l_{pc} \mathbf{n}_{cp} \, \psi_c^{n+\halb} \E_p^{n+\halb} \! \right)\!\!, \nonumber 
	\label{eqn.E.dp} 
	\\
	|\Omega_c|  \!  \left( \halb (\psi_c^{n+1})^2 - \halb (\psi_c^{n})^2 \right) \!\!\!\!\!\! &=& \!\!\!\!\!  - \Delta t \, \psi_c^{n+\halb} \, c_h  \sum_{p \in \Omega_c} l_{pc} \mathbf{n}_{pc} \cdot \E_p^{n+\halb}. \nonumber 
	\label{eqn.psi.dp} 
\end{eqnarray}	
Summing up each equation over all elements and summing up all equations yields 
\begin{eqnarray} 
&& \mathcal{E}^{n+1} - \mathcal{E}^n = 
\nonumber \\  
&& - \Delta t c_0 \sum_c \sum_{p \in \Omega_c} l_{pc} \mathbf{n}_{pc} \cdot \E_p^{n+\halb} \times \B_c^{n+\halb}
   + \Delta t c_0 \sum_p \sum_{c \in \Omega_p} l_{pc} \mathbf{n}_{cp} \cdot \B_c^{n+\halb} \times \E_p^{n+\halb}
\nonumber \\
&& - \Delta t c_h \sum_c \sum_{p \in \Omega_c} l_{pc} \mathbf{n}_{pc} \cdot \, \phi_p^{n+\halb} \B_c^{n+\halb}
   - \Delta t c_h \sum_p \sum_{c \in \Omega_p} l_{pc} \mathbf{n}_{cp} \cdot \, \B_c^{n+\halb} \phi_p^{n+\halb} 
\nonumber \\
&& - \Delta t c_h \sum_p \sum_{c \in \Omega_c} l_{pc} \mathbf{n}_{cp} \cdot \, \psi_c^{n+\halb} \E_p^{n+\halb}
   - \Delta t c_h \sum_c \sum_{p \in \Omega_p} l_{pc} \mathbf{n}_{pc} \cdot \, \E_p^{n+\halb} \psi_c^{n+\halb} = 0,
\end{eqnarray} 
since $\B_c^{n+\halb} \times \E_p^{n+\halb} = -\E_p^{n+\halb} \times \B_c^{n+\halb}$ and $\mathbf{n}_{cp}=-\mathbf{n}_{pc}$ and thus all pairwise interactions between $p$ and $c$ cancel. This proves total energy conservation and thus energy stability of the proposed numerical scheme. 
\end{proof}
\end{theorem}

\begin{theorem}
	\label{thm.ap} 	
	In the case of periodic boundaries and well-prepared initial data $\phi_p^0 = \phi^0=const$  and $\psi_c^0 = \psi^0=const$ the scheme \eqref{eqn.B.fd}-\eqref{eqn.psi.fd} is asymptotic-preserving (AP) in the sense that the discrete divergence $\nabla_c^p \cdot \E_p^{n+\halb} \to \mathcal{O}(\epsilon^2)$ and $\nabla_p^c \cdot \B_c^{n+\halb} \to \mathcal{O}(\epsilon^2)$ when $c_h \to \infty$ for fixed $c_0$, i.e. when $\epsilon = c_0/c_h \to 0$. 
	\begin{proof}
		Formal asymptotic expansion of the discrete solution in terms of $\epsilon = c_0/c_h$ yields 
		\begin{eqnarray}
			\B_c^n   &=& \B_{c,0}^n + \epsilon \B_{c,1}^n + \epsilon^2 \B_{c,2}^n + \mathcal{O}(\epsilon^3),  \nonumber \\ 
			\phi_p^n &=& \phi_{p,0}^n + \epsilon \phi_{p,1}^n + \epsilon^2 \phi_{p,2}^n + \mathcal{O}(\epsilon^3),  \nonumber \\ 
			\E_p^n &=& \E_{p,0}^n + \epsilon \E_{p,1}^n + \epsilon^2 \E_{p,2}^n + \mathcal{O}(\epsilon^3),  \nonumber \\ 
			\psi_c^n   &=& \psi_{c,0}^n + \epsilon \psi_{c,1}^n + \epsilon^2 \psi_{c,2}^n + \mathcal{O}(\epsilon^3).    
			\label{eqn.hilbert.expansion} 
		\end{eqnarray}	
	Insertion into the semi-implicit scheme \eqref{eqn.B.fd}-\eqref{eqn.psi.fd} provides the following relations up to order $\epsilon^0$: 		
	\begin{eqnarray}
		\epsilon^{-1} \,  \nabla_c^p \phi_{p,0}^{n+\halb} + 
		\epsilon^{0} \left( c_0^{-1} \frac{\B_{c,0}^{n+1} - \B_{c,0}^{n}}{\Delta t}  +    \nabla_c^p \times \E_{p,0}^{n+\halb} +   \nabla_c^p \phi_{p,1}^{n+\halb} \right)  &=& 0, 
		\label{eqn.B.fd.a} 
		\\ 
		\epsilon^{-1} \, \nabla_p^c \cdot \B_{c,0}^{n+\halb} + 
		\epsilon^{0} \left( c_0^{-1} \frac{\phi_{p,0}^{n+1} + \phi_{p,0}^{n}}{\Delta t} +
		\nabla_p^c \cdot \B_{c,1}^{n+\halb} \right) &=&  0,   
		\label{eqn.phi.fd.a} 
		\\
		\epsilon^{-1} \, \nabla_p^c \psi_{c,0}^{n+\halb} + 
		\epsilon^{0} \left( c_0^{-1} \frac{\E_{p,0}^{n+1} - \E_{p,0}^{n}}{\Delta t}  -   \nabla_p^c \times \B_{c,0}^{n+\halb} + \nabla_p^c \psi_{c,1}^{n+\halb} \right) & = & 0,
		\label{eqn.E.fd.a} 
		\\
		\epsilon^{-1} \, \nabla_c^p \cdot \E_{p,0}^{n+\halb}  + 
		\epsilon^{0} \left( c_0^{-1} \frac{\psi_{c,0}^{n+1} - \psi_{c,0}^{n}}{\Delta t} +  \nabla_c^p \cdot \E_{p,1}^{n+\halb} \right) &=& 0.
		\label{eqn.psi.fd.a} 
	\end{eqnarray}	
	From the leading order terms that scale with $\epsilon^{-1}$ in \eqref{eqn.phi.fd.a} and \eqref{eqn.psi.fd.a} we immediately obtain   
	\begin{equation}
	 	\nabla_p^c \cdot \B_{c,0}^{n+\halb} = 0, \qquad \textnormal{and} \qquad  \nabla_c^p \cdot \E_{p,0}^{n+\halb} = 0.
		\label{eqn.div.O0}	
	\end{equation}
	From the leading order terms in \eqref{eqn.B.fd.a} and \eqref{eqn.E.fd.a} one has $\nabla_c^p \phi_{p,0}^{n+\halb} = 0$ and $\nabla_p^c \psi_{c,0}^{n+\halb}= 0$. Since the initial data are assumed to be well-prepared, i.e. $\phi_p^0 = \phi^0$ and $\psi_c^0 = \psi^0$, i.e. $\nabla_c^p \phi_{p,0}^{n} = 0$ and $\nabla_p^c \psi_{c,0}^{n}= 0$ for $n=0$, we thus obtain that at leading order we have $\nabla_c^p \phi_{p,0}^{n+1} = 0$ and $\nabla_p^c \psi_{c,0}^{n+1}= 0$ for $n=0$. Via induction it is obvious that this holds also for all other $n>0$. This means that $\phi_{p,0}^{n} = \phi_{0}^{n}$ and $\psi_{c,0}^{n} = \psi_{0}^{n}$ are both constant in space and are only functions of time. Inserting these results into \eqref{eqn.phi.fd.a} and \eqref{eqn.psi.fd.a} and summing up over all elements yields 
	\begin{align}
	c_0^{-1} \sum_p |\Omega_p| \frac{\phi_{0}^{n+1} - \phi_{0}^{n}}{\Delta t} &
	=
	-  \sum_p |\Omega_p| \, \nabla_p^c \cdot \B_{c,1}^{n+\halb} = 
	-  \sum_p \sum_{c \in \Omega_p} l_{pc} \mathbf{n}_{cp} \cdot \B_{c,1}^{n+\halb} = 0,  
	\label{eqn.phi.fd.b} 
	\\
	c_0^{-1} \sum_c |\Omega_c|  \frac{\psi_{0}^{n+1} - \psi_{0}^{n}}{\Delta t} &=
	- \sum_c |\Omega_c| \, \nabla_c^p \cdot \E_{p,1}^{n+\halb} = 
	-  \sum_c \sum_{p \in \Omega_c} l_{pc} \mathbf{n}_{pc} \cdot \E_{p,1}^{n+\halb} = 0,
	\label{eqn.psi.fd.b} 
	\end{align}		
	where the last identities in the above equations follow from the fact that the order of the indices in the double sums can be rearranged thanks to the periodic boundary conditions and where \eqref{eqn.gauss.disc} has been used:
	\begin{eqnarray}
	 && \sum_p \sum_{c \in \Omega_p} l_{pc} \mathbf{n}_{cp} \cdot \B_{c,1}^{n+\halb} = 
	    \sum_c \sum_{p \in \Omega_c} l_{pc} \mathbf{n}_{cp} \cdot \B_{c,1}^{n+\halb} = 
	    \sum_c \B_{c,1}^{n+\halb} \cdot \sum_{p \in \Omega_c} l_{pc} \mathbf{n}_{cp} = 0, \qquad  
	\label{eqn.phi.fd.c} 
	\\
	&&  \sum_c \sum_{p \in \Omega_c} l_{pc} \mathbf{n}_{pc} \cdot \E_{p,1}^{n+\halb} = 
		\sum_p \sum_{c \in \Omega_p} l_{pc} \mathbf{n}_{pc} \cdot \E_{p,1}^{n+\halb} = 
		\sum_p \E_{p,1}^{n+\halb} \cdot \sum_{c \in \Omega_p} l_{pc} \mathbf{n}_{pc} = 0, \qquad 
	\label{eqn.psi.fd.c} 
\end{eqnarray}	
We thus have from \eqref{eqn.phi.fd.b} and \eqref{eqn.psi.fd.b} that 
	\begin{equation}\label{eqn.phi.psi.fd.d}
	c_0^{-1} |\Omega| \frac{\phi_{0}^{n+1} - \phi_{0}^{n}}{\Delta t} = 0,
	\qquad 
	c_0^{-1} |\Omega| \frac{\psi_{0}^{n+1} - \psi_{0}^{n}}{\Delta t} = 0, 
\end{equation}	
with $|\Omega| = \sum_p |\Omega_p| = \sum_c |\Omega_c|$ the area of the computational domain. Hence, at leading order the discrete time derivatives of 
the discrete scalar fields $\phi$ and $\psi$ vanish. Substituting this result with $\phi_{p,0}^{n} = \phi_{0}^{n}$ and $\psi_{c,0}^{n} = \psi_{0}^{n}$ back into \eqref{eqn.phi.fd.a} and \eqref{eqn.psi.fd.a} yields
	\begin{equation}
	\nabla_p^c \cdot \B_{c,1}^{n+\halb} = 0, \qquad \textnormal{and} \qquad  \nabla_c^p \cdot \E_{p,1}^{n+\halb} = 0.
	\label{eqn.div.O1}	
\end{equation}
With the expansion \eqref{eqn.hilbert.expansion} we thus obtain the sought result 
	\begin{equation}
	\nabla_p^c \cdot \B_{c}^{n+\halb} = \mathcal{O}(\epsilon^2), \qquad \textnormal{and} \qquad  
	\nabla_c^p \cdot \E_{p}^{n+\halb} = \mathcal{O}(\epsilon^2).
	\label{eqn.div.O1}	
\end{equation}
for $\epsilon \to 0$. 
	\end{proof}
\end{theorem}
Furthermore, for initial data that are compatible with the vacuum Maxwell equations, i.e. that satisfy $\nabla \cdot \B = 0$, $\nabla \cdot  \E = 0$ and $\phi = \psi = 0$ at $t=0$ via the discrete identity \eqref{eqn.div.rot} it is trivial to prove that the proposed structure-preserving staggered semi-discrete scheme \eqref{eqn.B.fd}-\eqref{eqn.psi.fd} preserves the divergence-free condition of the magnetic and electric field exactly at the discrete level for all times. 

\section{Numerical results}
\label{sec.results}

In all of the following tests we set $c_0=1$ and use the nine-stage seventh-order accurate Runge-Kutta scheme of Butcher \cite{butcher} to integrate the semi-discrete HTC scheme \eqref{eqn.flux2dfv} in time. Furthermore, in all the tests we use periodic boundary conditions everywhere.  

\subsection{Numerical convergence study} 
We first carry out a systematic numerical convergence study of the two numerical schemes presented in this paper. For this purpose we consider the following initial data of a planar wave travelling in the direction $\normal = (1,1,0)$. The initial condition is given by 
\begin{eqnarray}
	& \B(x, y)  = \mathbf{B_0} \sin\left( \pi (x-y) \right), \qquad & \phi(x,y) = 0.25\sin \left(\pi (x-y)\right),	\\
	& \E(x, y)  = \mathbf{E_0} \sin\left( \pi (x-y) \right), \qquad & \psi(x,y) = 0.5  \sin\left(\pi (x-y)\right), 
\end{eqnarray}
with $\mathbf{B_0}=(0.25b, -0.25b, 1)^T $, $ \mathbf{E_0}=(1.5b, 0.5b, 0)^T $ and $b=\sqrt{2}/2$.
The parameters of the model are chosen to be $ c_h=1.0 $ and $ c_0=1.0 $. The computational domain $\Omega=[-1,1]^2$ is discretized with a sequence of successively refined uniform grids composed of $ N\times N $ elements and simulations are run for one period until a final time of $t=\sqrt{2}$, when the exact solution of the problem coincides again with the initial condition. The semi-discrete HTC scheme employs a $7^{\text{th}} $-order accurate Runge-Kutta method for time discretization \cite{butcher}. In both schemes, the time step is forced to obey the $\mathrm{CFL}$ stability condition with $\mathrm{CFL} =0.9 $, i.e. $\Delta t = \mathrm{CFL}/(c_0/\Delta x + c_0/\Delta y)$. Tables \ref{tab:HTC_err} and \ref{tab:SIMM_err} show the $L^2$ errors of all relevant quantities at the final time depending on the chosen grid spacing. As one can observe, both methods converge with second order of accuracy in both space and time. 

\begin{table}[!h]
	\centering
	\renewcommand{\arraystretch}{1.1}
	\begin{tabular}{l|cccc|ccc}
		\hline
		&\multicolumn{4}{c|}{$L^2$ errors} & \multicolumn{3}{c}{Numerical convergence} \\				
		$N$	&20		&40		&80		&160 & \multicolumn{3}{c}{order}		\\
		\hline		
		$B_1$	&2.57$\cdot10^{-2}$	&6.45$\cdot10^{-3}$	&1.61$\cdot10^{-3}$	&4.04$\cdot10^{-4}$	&1.99	&2.00	&2.00	\\
		$B_2$	&2.57$\cdot10^{-2}$	&6.45$\cdot10^{-3}$	&1.61$\cdot10^{-3}$	&4.04$\cdot10^{-4}$	&1.99	&2.00	&2.00	\\
		$B_3$	&1.45$\cdot10^{-1}$	&3.65$\cdot10^{-2}$	&9.13$\cdot10^{-3}$	&2.28$\cdot10^{-3}$	&1.99	&2.00	&2.00	\\
		$\phi$	&3.63$\cdot10^{-2}$	&9.12$\cdot10^{-3}$	&2.28$\cdot10^{-3}$	&5.71$\cdot10^{-4}$	&1.99	&2.00	&2.00	\\
		$E_1$	&1.54$\cdot10^{-1}$	&3.87$\cdot10^{-2}$	&9.69$\cdot10^{-3}$	&2.42$\cdot10^{-3}$	&1.99	&2.00	&2.00	\\
		$E_2$	&5.14$\cdot10^{-2}$	&1.29$\cdot10^{-2}$	&3.23$\cdot10^{-3}$	&8.07$\cdot10^{-4}$	&1.99	&2.00	&2.00	\\
		$\psi$	&7.27$\cdot10^{-2}$	&1.82$\cdot10^{-2}$	&4.57$\cdot10^{-3}$	&1.14$\cdot10^{-3}$	&1.99	&2.00	&2.00	\\
		\hline
	\end{tabular}
	\caption{$L^2$ error norms and corresponding empirical convergence rates for the planar wave travelling in the direction $\normal = (1,1,0)$ obtained with the semi-discrete HTC scheme on uniform grids composed of $N \times N$  elements. }
	\label{tab:HTC_err}
\end{table}

\begin{table}[!h]
	\centering
	\renewcommand{\arraystretch}{1.1}
	\begin{tabular}{l|cccc|ccc}
		\hline
		&\multicolumn{4}{c|}{$L^2$ errors} & \multicolumn{3}{c}{Numerical convergence} \\				
		$N$	&20		&40		&80		&160 & \multicolumn{3}{c}{order}		\\
		\hline		
		$B_1$	&3.06$\cdot10^{-2}$	&7.74$\cdot10^{-3}$	&1.94$\cdot10^{-3}$	&4.85$\cdot10^{-4}$	&1.98	&2.00	&2.00	\\
		$B_2$	&3.06$\cdot10^{-2}$	&7.74$\cdot10^{-3}$	&1.94$\cdot10^{-3}$	&4.85$\cdot10^{-4}$	&1.98	&2.00	&2.00	\\
		$B_3$	&1.73$\cdot10^{-1}$	&4.38$\cdot10^{-2}$	&1.10$\cdot10^{-2}$	&2.75$\cdot10^{-3}$	&1.98	&2.00	&2.00	\\
		$\phi$	&4.33$\cdot10^{-2}$	&1.09$\cdot10^{-2}$	&2.74$\cdot10^{-3}$	&6.86$\cdot10^{-4}$	&1.98	&2.00	&2.00	\\
		$E_1$	&1.84$\cdot10^{-1}$	&4.64$\cdot10^{-2}$	&1.16$\cdot10^{-2}$	&2.91$\cdot10^{-3}$	&1.98	&2.00	&2.00	\\
		$E_2$	&6.12$\cdot10^{-2}$	&1.55$\cdot10^{-2}$	&3.88$\cdot10^{-3}$	&9.71$\cdot10^{-4}$	&1.98	&2.00	&2.00	\\
		$\psi$	&8.65$\cdot10^{-2}$	&2.19$\cdot10^{-2}$	&5.49$\cdot10^{-3}$	&1.37$\cdot10^{-3}$	&1.98	&2.00	&2.00	\\
		\hline
	\end{tabular}
	\caption{$L^2$ error norms and corresponding empirical convergence rates for the planar wave travelling in the direction $\normal = (1,1,0)$ problem obtained with the fully-discrete semi-implicit scheme on uniform grids composed of $N \times N$  elements. }
	\label{tab:SIMM_err}
\end{table}
In addition Figure \ref{fig:ErrEn} shows the temporal evolution of the relative total energy error $\mathcal{E}^n/\mathcal{E}^0 - 1$ for both schemes and for all meshes considered in the convergence study. As one can observe, and as expected, total energy is preserved up to machine precision in all cases.  
\begin{figure}[!h]
	\centering
	\includegraphics[width=0.45\textwidth]{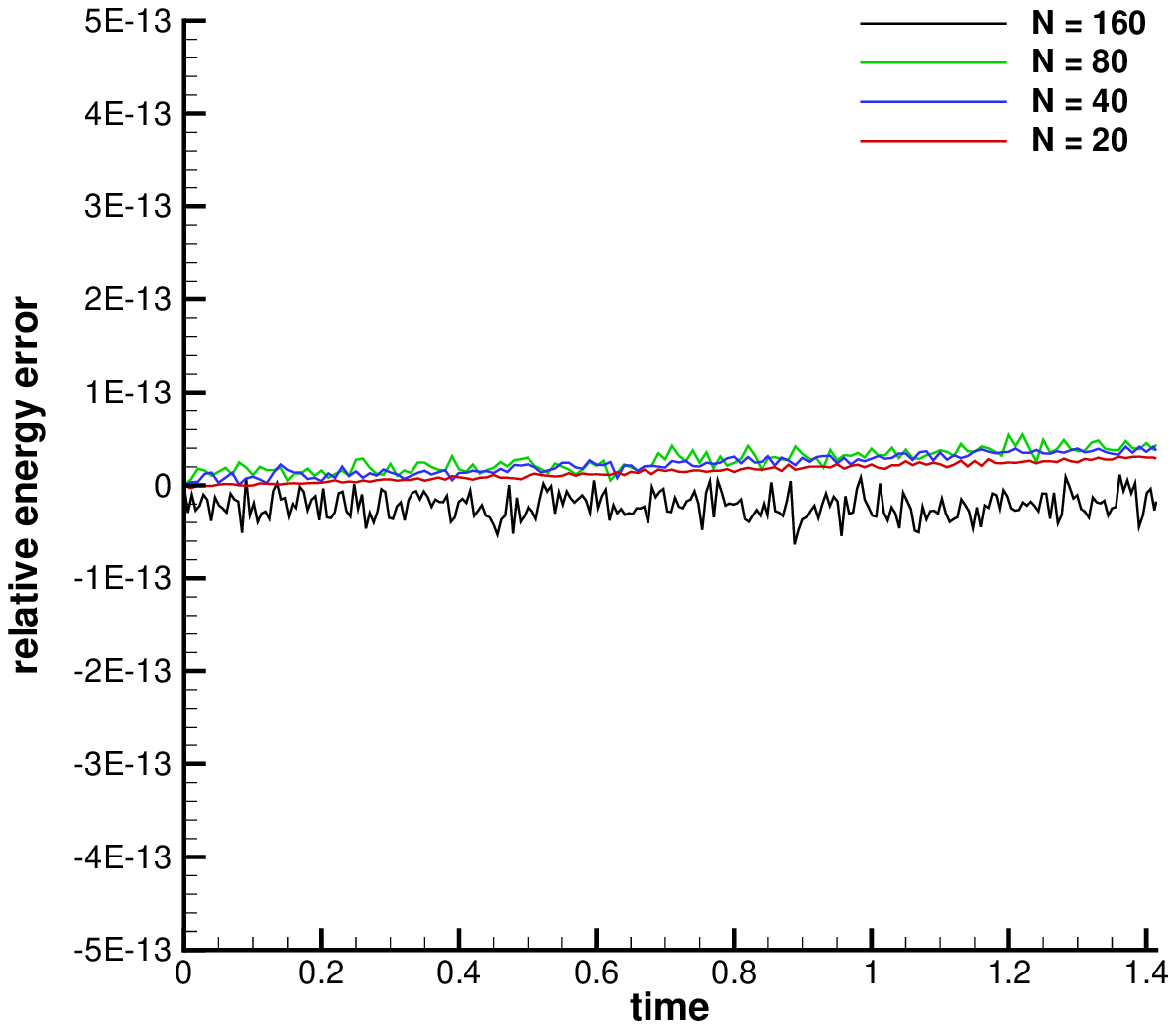}\quad \includegraphics[width=0.45\textwidth]{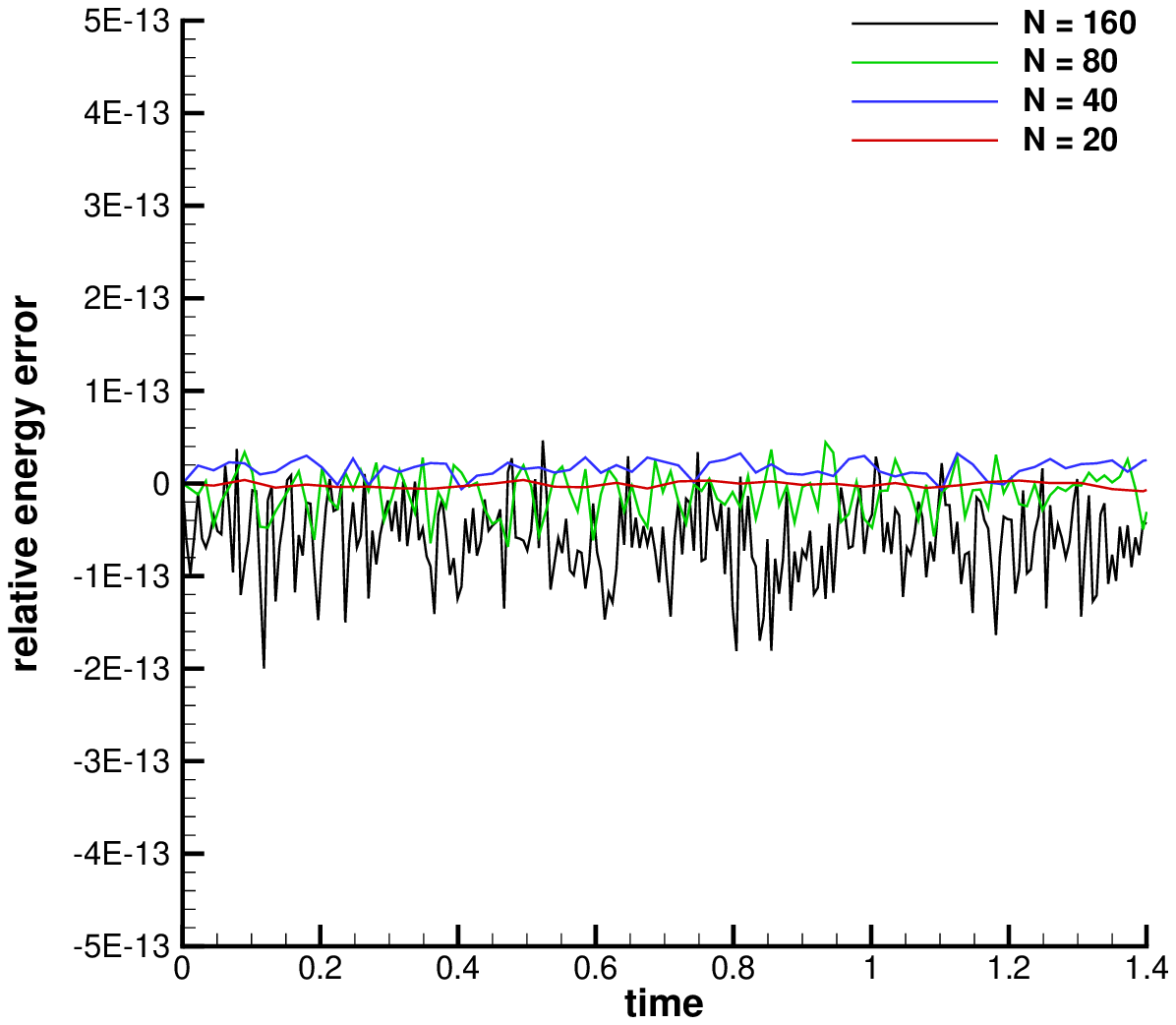}
	\caption{Temporal evolution of the relative total energy error $\mathcal{E}^n/\mathcal{E}^0 - 1$ for a sequence of successively refined uniform grids composed of $ N\times N $ elements obtained with the semi-discrete HTC scheme (left panel) as well as the fully-discrete semi-implicit scheme (right panel).}
	\label{fig:ErrEn}
\end{figure}

\subsection{Discrete total energy conservation and exact divergence preservation} 
Both the semi-discrete HTC scheme and the fully discrete semi-implicit scheme are exactly energy conservative. Furthermore, the structure-preserving staggered semi-implicit scheme is also exactly compatible with the divergence free condition of the electric and the magnetic field provided the initial data are compatible with those of the vacuum Maxwell equations. To test these properties of the schemes, we therefore solve the following test problem, given by the initial data 
\begin{eqnarray}
	& \B(x, y)  = \mathbf{B_0} \exp \left( -\halb \x^2/\sigma^2 \right), \qquad & \phi(x,y) = \phi_0 \exp\left(-\halb \x^2/\sigma^2 \right),	\nonumber \\
	& \E(x, y)  = \mathbf{E_0} \exp \left( -\halb \x^2/\sigma^2 \right), \qquad & \psi(x,y) = \psi_0 \exp\left(-\halb \x^2/\sigma^2 \right), 
	\label{eqn.ic.gauss}
\end{eqnarray}
Simulations are run for all schemes in the domain $\Omega=[-1,1]^2$ until a final time of $t=10$, which is substantially larger than the final time of the previous test. We propose two different test cases, T1 and T2. In the first test (T1) we choose initial data that are compatible with the original vacuum Maxwell equations. We therefore set $\mathbf{B_0} = \mathbf{E_0} = (0,0,10^{-2})$ and $\phi_0 = \psi_0 = 0$.  
In the second test (T2) we choose initial data that are deliberately not compatible with the original vacuum Maxwell equations, setting 
$\mathbf{B_0} = \mathbf{E_0} = (0.25 \cdot 10^{-2}, 0, 10^{-2})$ and $\phi_0 = \psi_0 = 0.5 \cdot 10^{-2}$. In both tests the half width is chosen as $\sigma=0.2$. As in the previous test case, we set $c_0=1$, $c_h=1$. 

The left panel of Figure \ref{fig:ErrEn.Gaussian} shows the temporal evolution of the relative total energy error $\mathcal{E}^n/\mathcal{E}^0 - 1$ provided by the fully-discrete semi-implicit scheme (SIMM) on a fixed mesh composed of $50 \times 50$ elements for both tests T1 and T2, and time step $\Delta t = \mathrm{CFL}/(c_0/\Delta x + c_0/\Delta y)$, where $ \mathrm{CFL}=0.9 $. 
In addition, test T2 has been also solved employing the semi-discrete HTC scheme on a fixed mesh composed of $160 \times 160$ elements evolved in time according to a seventh-order accurate Runge-Kutta method with a time step restricted by the $\mathrm{CFL}$ stability condition with $ \mathrm{CFL}=0.6 $.
As one can observe, the total energy is preserved up to machine precision in all the cases. The right panel of Figure \ref{fig:ErrEn.Gaussian} shows the time evolution of the divergence errors of the magnetic and the electric field obtained with the SIMM scheme for test T1 . As expected, the divergence errors remain of the order of machine precision for all times. 
We now repeat test case T2 also for a nonlinear total energy potential. We choose a total energy density given by
\begin{equation}
	\mathcal{E} = c_0\exp\left(\halb \B^2 \right) +  c_0\exp\left(\halb \E^2 \right) + \frac{c_h^2}{c_0} \exp\left(\halb \phi^2 \right)  + \frac{c_h^2}{c_0} \exp\left(\halb \psi^2 \right).  
	\label{eqn.totalenergy_nonlin} 
\end{equation}
The simulation is run for the semi-discrete HTC scheme employing a nine stages seventh-order accurate Runge-Kutta method for time discretization, where the time step is restricted by a CFL stability condition with CFL$ =0.9 $.
The temporal evolution of the relative total energy error $\mathcal{E}^n/\mathcal{E}^0 - 1$ on a fixed mesh composed of $160 \times 160$ elements is reported in the left panel of Figure \ref{fig:ErrEn.Gaussian}. As expected, total energy is conserved up to machine precision. 

\begin{figure}[!h]
	\centering
	\includegraphics[width=0.45\textwidth]{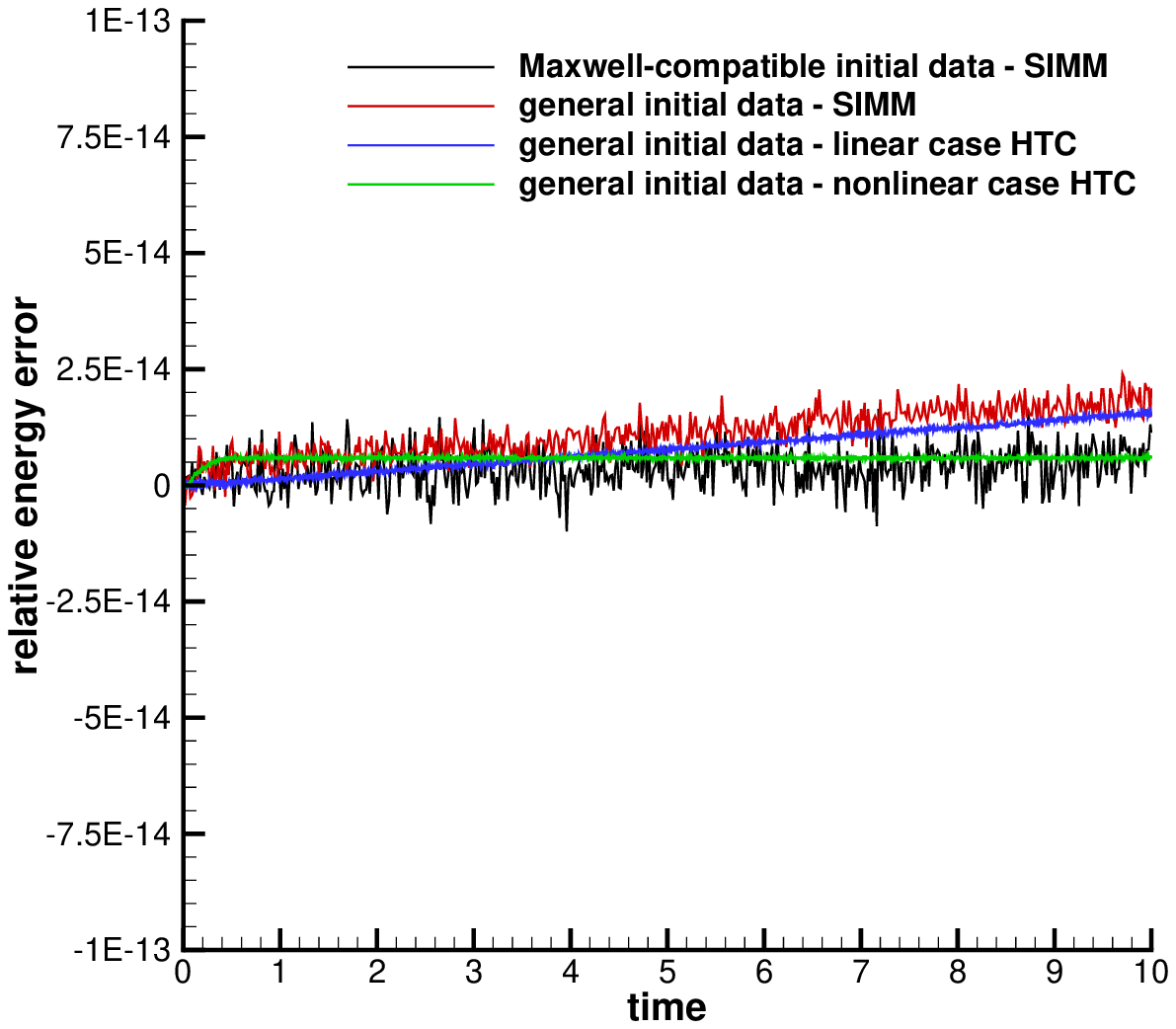}  \quad 
	\includegraphics[width=0.45\textwidth]{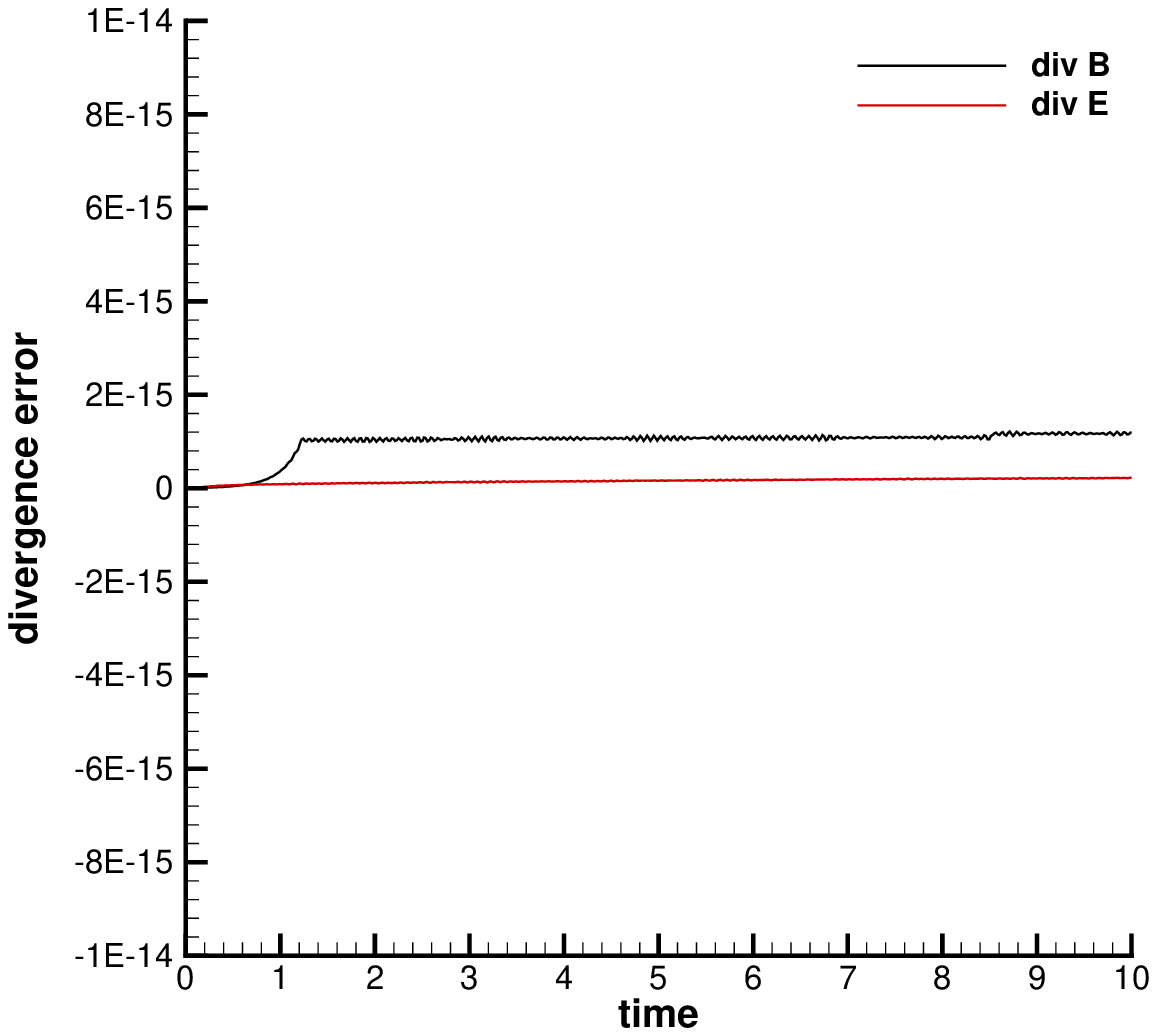}
	\caption{Left panel: Temporal evolution of the relative total energy error $\mathcal{E}^n/\mathcal{E}^0 - 1$ for initial data that is compatible with the vacuum Maxwell equations (black line), for general initial data employing the semi-implicit scheme (red line) as well as the HTC scheme (blue line) and for general initial data in case the nonlinear total energy potential \eqref{eqn.totalenergy_nonlin} is used and the problem solved with the HTC scheme (green line). Right panel: time series of the divergence errors of the magnetic and electric field for the SIMM scheme and for initial data that is compatible with the vacuum Maxwell equations. }
	\label{fig:ErrEn.Gaussian}
\end{figure}

\subsection{Discrete asymptotic preserving property for $c_h \to \infty$} 
In this last test we show numerical evidence for the fact that the staggered semi-implicit structure-preserving scheme is also asymptotic preserving (AP) in the limit $c_h \to \infty$, i.e. $\epsilon = c_0/c_h \to 0$. For this purpose we use again the same initial condition \eqref{eqn.ic.gauss} as in the previous section, but with amplitudes $\mathbf{B_0} = \mathbf{E_0} = (10^{-4}, 0, 10^{-2})$ and $\phi_0 = \psi_0 = 0$, to satisfy the hypothesis of Theorem \ref{thm.ap} but in which the magnetic and electric field are initially \textit{not} divergence-free. We now run a set of simulations on the domain $\Omega=[-1,+1]^2$ discretized with $40 \times 40$ elements until a final time of $t=0.1$ with a fixed time step of size $\Delta t = 10^{-2}$. We choose different cleaning speeds $c_h$ in the range between $10^2$ and $10^5$ and compute the discrete divergence of the magnetic and electric fields in terms of $c_h$. Table \ref{tab:SIMM_AP} clearly shows that the divergence errors converge quadratically in the parameter $\epsilon = c_0 / c_h$, as expected and as predicted by Theorem \ref{thm.ap}.  

\begin{table}[!h]
	\centering
	\renewcommand{\arraystretch}{1.15}
	\begin{tabular}{cc|cc|cc}
		\hline
		\multicolumn{2}{c|}{$c_0=1$} &\multicolumn{2}{c|}{$L^2$ divergence errors} & \multicolumn{2}{c}{Convergence order} \\				
		$c_h$	& $\epsilon$		&  $\nabla_p^c \cdot \B_c^{n+\halb}$		&$\nabla_c^p \cdot \E_p^{n+\halb}$	 & $\nabla_p^c \cdot \B_c^{n+\halb}$		&$\nabla_c^p \cdot \E_p^{n+\halb}$ 		\\
		\hline		
		$10^2$	& $10^{-2}$ & 3.831380$\cdot10^{-5}$      & 3.831579$\cdot10^{-5}$	&     & 	 \\
		$10^3$	& $10^{-3}$ & 3.569500$\cdot10^{-6}$      & 3.569623$\cdot10^{-6}$	& 1.0 & 1.0  \\
		$10^4$  & $10^{-4}$ & 4.351311$\cdot10^{-8}$      & 4.351523$\cdot10^{-8}$	& 1.9 & 1.9  \\
		$10^5$	& $10^{-5}$ & 4.368280$\cdot10^{-10}$     & 4.358525$\cdot10^{-10}$	& 2.0 & 2.0	 \\
		\hline
	\end{tabular}
	\caption{$L^2$ error norms at $t=0.1$ and corresponding convergence rates of the divergence errors of the magnetic and the electric field in terms of the dimensionless parameter $\epsilon=c_0/c_h$ on a fixed mesh of $40 \times 40$ elements obtained with the fully-discrete semi-implicit scheme for increasing values of the cleaning speed $c_h$. In all simulations we have set $c_0=1$.}
	\label{tab:SIMM_AP}
\end{table}

\section{Conclusion}
\label{sec.conclusion}

In this paper we have provided an original variational derivation of the augmented hyperbolic GLM divergence cleaning system of Munz \textit{et al.} \cite{MunzCleaning}, which is a very interesting symmetric hyperbolic and thermodynamically compatible (SHTC) system. We furthermore have provided a formal asymptotic analysis for the limit $\epsilon = c_0 / c_h \to 0$ as well as its extension to general nonlinear energy potentials and to the case of special relativity. We also have established a direct connection to Hamiltonian mechanics and have found the associated Poisson bracket. To the best of our knowledge, the Maxwell-Munz system studied in this paper goes beyond the existing set of known SHTC systems. Concerning its numerical discretization we have discussed two different types of exactly structure-preserving schemes. A first set of semi-discrete cell-centered HTC finite volume schemes that employ collocated grids and that are exactly compatible with the total energy conservation law, but which do in general not respect the basic identities of vector calculus. We then have also proposed a second fully-discrete method that uses vertex-based staggered grids, employs compatible mimetic finite difference operators for the discrete gradients so that all basic vector calculus identities are exactly satisfied also at the discrete level. Furthermore, the second scheme also conserves total energy exactly at the fully-discrete level. It furthermore is asymptotically consistent with the limit when $\epsilon = c_0 / c_h \to 0$, as shown by numerical evidence. We also have provided some numerical results concerning the extension of both schemes to the fully nonlinear case. In all cases the numerical results confirm the theoretical expectations. 

Future work will concern the extension of the staggered semi-implicit HTC schemes introduced in this work to other more complex SHTC systems, including also numerical dissipation and entropy production, which were both not present in the methods developed in this paper. We also plan to apply the present schemes to the system of Buchman \textit{et al.} in the context of new structure-preserving schemes for numerical general relativity.  

{\small
\section*{Acknowledgements}
	This work was financially supported by the Italian Ministry of Education, University 
	and Research (MIUR) in the framework of the PRIN 2022 project \textit{High order structure-preserving semi-implicit schemes for hyperbolic equations}, the PRIN 2022 project No. 2022N9BM3N \textit{Efficient numerical schemes and optimal control methods for time-dependent partial differential equations} and via the  Departments of Excellence  Initiative 2018--2027 attributed to DICAM of the University of Trento (grant L. 232/2016). 
	MD and AL were also funded by the European Union Next Generation EU project PNRR Spoke 7 CN HPC
	and by the European Research Council (ERC) under the European Union’s Horizon 
	2020 research and innovation programme, Grant agreement No. ERC-ADG-2021-101052956-BEYOND. 
	Views and opinions expressed are however those of the author(s) only and
	do not necessarily reflect those of the European Union or the European Research
	Council. Neither the European Union nor the granting authority can be held
	responsible for them. 
	MD, AL and IP are members of the Gruppo Nazionale Calcolo Scientifico-Istituto Nazionale di Alta Matematica (GNCS-INdAM). 
	\\
	\textbf{This paper is dedicated to Prof. Claus-Dieter Munz on the occasion of his $70^{\textnormal{th}}$ birthday}. 

}

\section*{Appendix}
This appendix aims to verify the discrete vector calculus identities \eqref{eqn.rot.grad} and \eqref{eqn.div.rot}.

To simplify the computation, we restrict the discussion to two-dimensional case, i.e. we assume that $ \frac{\partial}{\partial z} $ vanishes for all fields and thus, we assume a two-dimensional physical domain covered by a set of equidistant and non-overlapping Cartesian control volumes $ \Omega_{i,j} = [x_{i-\halb},x_{i+\halb}]\times[y_{j-\halb}, y_{j+\halb}]$ with uniform mesh spacings of size $ \Delta x = x_{i+\halb}-x_{i-\halb} $ and $ \Delta y = y_{j+\halb}-y_{j-\halb} $ in $ x$ and $ y $ direction, respectively, where $ x_{i\pm\halb} =x_i\pm\frac{\Delta x}{2}$ and $ y_{j\pm\halb} =y_j\pm\frac{\Delta y}{2}$.
For a schematic view of the grid, we refer again to Fig.\ref{fig:Grid} where $ c $ represents a pair of indices of the type $ (i,j) $, while $ p $ is of the form $ (i\pm\halb, j\pm\halb) $.

Let us introduce a scalar field in the barycenters of the primal control volume $ \Omega_{i,j} $, i.e. $ \phi_{i,j} $, then the discrete nabla operator \eqref{eqn.nablapc} defines a discrete gradient of $ \phi $ in all vertices of the primal mesh, i.e.
\begin{equation}
	\nabla_{i+\halb,j+\halb}^c \phi_c = 
	\begin{pmatrix}
		\displaystyle\frac{1}{2\Delta x}\left(\phi_{i+1,j+1}+\phi_{i+1,j}-\phi_{i,j+1}-\phi_{i,j}\right)\\
		\displaystyle\frac{1}{2\Delta y}\left(\phi_{i+1,j+1}+\phi_{i,j+1}-\phi_{i+1,j}-\phi_{i,j}\right)\\
		0 \\
	\end{pmatrix}\!.
\end{equation}
Applying the discrete curl operator to the last term, it is straightforward to verify that we obtain a null identity, that is 
\begin{align}
	(\nabla_{i,j}^p \times &\nabla_p^c \phi_c) \cdot \mathbf{e_z} = \nonumber\\
		&+\frac{\phi_{i+1,j+1}+\phi_{i,j+1}-\phi_{i+1,j}-\phi_{i,j}+\phi_{i+1,j}+\phi_{i,j}-\phi_{i+1,j+1}-\phi_{i,j-1}}{4\Delta x\Delta y}\nonumber\\
		&-\frac{\phi_{i,j-1}+\phi_{i-1,j+1}-\phi_{i,j}-\phi_{i-1,j}+\phi_{i,j}+\phi_{i-1,j}-\phi_{i,j-1}-\phi_{i-1,j-1}}{4\Delta x\Delta y}\nonumber\\
		&-\frac{\phi_{i+1,j+1}+\phi_{i+1,j}-\phi_{i,j+1}-\phi_{i,j}+\phi_{i,j+1}+\phi_{i,j}-\phi_{i-1,j+1}-\phi_{i-1,j}}{4\Delta x\Delta y}\nonumber\\
		&+\frac{\phi_{i+1,j}+\phi_{i+1,j-1}-\phi_{i,j}-\phi_{i,j-1}+\phi_{i,j}+\phi_{i,j-1}-\phi_{i-1,j}-\phi_{i-1,j-1}}{4\Delta x\Delta y}\nonumber\\
		&=0,
\end{align}
where $ \mathbf{e_z} =(0, 0, 1) $ is  the third unit vector  in Cartesian coordinates.
We have therefore shown that the continuous identity $\nabla \times \nabla \phi = 0$ has
a discrete analog represented by 
the relation 
\begin{equation}
	\nabla_c^p \times \nabla_p^c \, \phi_c = 0\,.
\end{equation}
In the same way one can prove the dual identity $ \nabla_p^c \times \nabla_c^p \, \phi_p = 0.  $

Similarly, we introduce a vector field $ \A_{i,j} $ defined in the $ \Omega_{i,j} $ barycenter and we compute $ \nabla_p^c \times\A_c $, that is
\begin{multline}
	\nabla_{i+\halb,j+\halb}^c \times\mathbf{A}_c =\\
	\begin{pmatrix}
		\displaystyle\frac{A_{3,i+1,j+1}+A_{3,i,j+1}-A_{3,i+1,j}-A_{3,i,j}}{2\Delta y}\\
		-\displaystyle\frac{A_{2,i+1,j+1}+A_{2,i+1,j}-A_{2,i,j+1}-A_{2,i,j}}{2\Delta x}\\
		\displaystyle\frac{ A_{2,i+1,j+1}+A_{2,i+1,j}-A_{2,i,j+1}-A_{2,i,j}}{2\Delta x}-\displaystyle\frac{A_{1,i+1,j+1}+A_{1,i,j+1}-A_{1,i+1,j}-A_{1,i,j}}{2\Delta y} \\
	\end{pmatrix}.
\end{multline}
Applying the discrete divergence to the discrete curl of $\A$, it is possible to verify that the continuous identity $\nabla \cdot \nabla \times \mathbf{A} = 0 $ holds also at the discrete level. Indeed, we have
\begin{align}
	\nabla_{i,j}^p \cdot &\nabla_p^c \times \mathbf{A}_c = \nonumber\\
	&+\frac{A_{3,i+1,j+1}+A_{3,i,j+1}-A_{3,i+1,j}-A_{3,i,j}+A_{3,i+1,j}+A_{3,i,j}-A_{3,i+1,j-1}-A_{3,i,j-1}}{4\Delta x\Delta y}\nonumber\\
	&-\frac{A_{3,i,j+1}+A_{3,i-1,j+1}-A_{3,i,j}-A_{3,i-1,j}+A_{3,i,j}+A_{3,i-1,j}-A_{3,i,j-1}-A_{3,i-1,j-1}}{4\Delta x\Delta y}\nonumber\\
	&-\frac{A_{3,i+1,j+1}+A_{3,i+1,j}-A_{3,i,j+1}-A_{3,i,j}+A_{3,i,j+1}+A_{3,i,j}-A_{3,i-1,j+1}-A_{3,i-1,j}}{4\Delta x\Delta y}\nonumber\\
	&+\frac{A_{3,i+1,j}+A_{3,i+1,j-1}-A_{3,i,j}-A_{3,i,j-1}+A_{3,i,j}+A_{3,i,j-1}-A_{3,i-1,j}-A_{3,i-1,j-1}}{4\Delta x\Delta y}\nonumber\\
	&=0.
\end{align}
In the same way one can prove the dual identity $ \nabla_p^c \cdot \nabla_c^p \times \A_p = 0.  $

\printbibliography

\end{document}